\documentclass[11pt]{article}
\usepackage[T1]{fontenc}
\usepackage[latin9]{inputenc}
\usepackage{geometry}
\usepackage{mathrsfs}
\usepackage{amsthm}
\usepackage{amsmath}
\usepackage{amssymb}
\usepackage[unicode=true,pdfusetitle,
 bookmarks=true,bookmarksnumbered=false,bookmarksopen=false,
 breaklinks=false,pdfborder={0 0 1},backref=false,colorlinks=false]
 {hyperref}

\makeatletter
\numberwithin{equation}{section}
  \theoremstyle{plain}
  \newtheorem*{conjecture*}{\protect\conjecturename}
  \theoremstyle{remark}
  \newtheorem*{rem*}{\protect\remarkname}
\theoremstyle{plain}
\newtheorem{thm}{\protect\theoremname}[section]
  \theoremstyle{remark}
  \newtheorem{rem}[thm]{\protect\remarkname}
  \theoremstyle{plain}
  \newtheorem{cor}[thm]{\protect\corollaryname}
  \theoremstyle{plain}
  \newtheorem{lem}[thm]{\protect\lemmaname}
  \theoremstyle{definition}
  \newtheorem{example}[thm]{\protect\examplename}
  \theoremstyle{definition}
  \newtheorem{defn}[thm]{\protect\definitionname}
  \theoremstyle{plain}
  \newtheorem{prop}[thm]{\protect\propositionname}
  \theoremstyle{plain}
  \newtheorem{fact}[thm]{\protect\factname}
  \theoremstyle{remark}
  \newtheorem{claim}[thm]{\protect\claimname}

\date{}
\usepackage{mathrsfs}

\theoremstyle{plain}
\newtheorem{mythm}{\protect\theoremname}

\newtheorem{thmrpt}{\protect\theoremname}
\renewcommand*{\thethmrpt}{\Alph{thmrpt}}

\let\stdpart\part
\renewcommand*{\part}{\clearpage\stdpart}

\usepackage{cite}

\DeclareMathOperator{\pf}{Pf}
\DeclareMathOperator{\lcm}{lcm}

\makeatother

  \providecommand{\claimname}{Claim}
  \providecommand{\conjecturename}{Conjecture}
  \providecommand{\corollaryname}{Corollary}
  \providecommand{\definitionname}{Definition}
  \providecommand{\examplename}{Example}
  \providecommand{\factname}{Fact}
  \providecommand{\lemmaname}{Lemma}
  \providecommand{\propositionname}{Proposition}
  \providecommand{\remarkname}{Remark}
\providecommand{\theoremname}{Theorem}

\begin{document}

\global\long\def\A{\mathcal{A}}
\global\long\def\B{\mathcal{B}}
\global\long\def\W{\mathcal{W}}
\global\long\def\Y{\mathcal{Y}}
\global\long\def\V{\mathcal{V}}
\global\long\def\G{\mathcal{G}}

\global\long\def\solspace{\mathcal{H}}

\global\long\def\ZZ{\mathbb{Z}}
\global\long\def\NN{\mathbb{N}}
\global\long\def\KK{\mathbb{K}}
\global\long\def\CC{\mathbb{C}}
\global\long\def\RR{\mathbb{R}}

\global\long\def\F{\RR}
\global\long\def\Af{\mathbb{A}}

\global\long\def\z{\zeta}
\global\long\def\e{\epsilon}
\global\long\def\a{\alpha}
\global\long\def\b{\beta}
\global\long\def\ph{\varphi}
\global\long\def\ga{\gamma}
\global\long\def\dl{\delta}

\global\long\def\l{\ell}

\global\long\def\perm{\mbox{perm}}
\global\long\def\lcm{\mbox{lcm}}
 \global\long\def\haf{\mbox{hf}}
\global\long\def\pitz{\mathbf{P}}

\global\long\def\var{{\cal Y}}
\global\long\def\lvar{{\cal W}}
\global\long\def\inpol{\Delta}
\global\long\def\odd#1{{#1}^{\ast}}

\global\long\def\vpf{\mathcal{F}}
\global\long\def\vdet{{\cal D}}
\global\long\def\vrec{{\cal R}}
\global\long\def\vper{{\cal P}}
\global\long\def\vhf{\mathcal{H}}
\global\long\def\vquad{{\cal Q}}
\global\long\def\vgen{\mathcal{X}}

\global\long\def\asym{\mathscr{A}}
\global\long\def\sym{\mathscr{S}}
\global\long\def\con{{\cal C}}

\global\long\def\matA{a}
 \global\long\def\matB{b}

\title{Prime Points in Orbits: 
Some Instances \\ of the Bourgain-Gamburd-Sarnak Conjecture}

\author{Tal Horesh and Amos Nevo}
\maketitle
\begin{abstract}
We use Vaughan's variation on Vinogradov's three-primes theorem to
prove Zariski-density of prime points in several infinite families
of hypersurfaces, including level sets of some quadratic forms, the Permanent
polynomial, and the defining polynomials of some pre-homogeneous vector spaces. Three of these families are instances of a conjecture
by Bourgain, Gamburd and Sarnak regarding prime points in orbits of
simple algebraic groups. Our approach is based on the formulation of a general condition on the defining polynomial of a hypersurface, which suffices to guarantee  that Zariski-density
of prime points is equivalent to the existence of an odd point.
\end{abstract}
\tableofcontents{}

\section{Introduction and statement of the main results\label{sec: Introduction and examples}}

\textit{Dirichlet's Theorem on Arithmetic Progressions} states that
if two integers $\a$ and $\b$ are co-prime, then there are infinitely
many primes in the arithmetic progression $\a,\,\a+\b,\,\a+2\b,...$.
In other words, there are infinitely many primes that are congruent
to $\a$ modulo $\b$. In his survey \cite{Sarnak_lecture}, Sarnak
has suggested a new perspective on Dirichlet's theorem: the additive
linear algebraic group $\mathbb{G}_{a}$ over $\RR$ acts on the space
$\Af^{1}=\RR^{1}$ by translations, namely $g\cdot x=x+g$ for $x\in\Af^{1}$
and $g\in\mathbb{G}_{a}$. Write $\mathbb{G}_{a}\left(\ZZ\right)$
for the lattice subgroup of integer points in $\mathbb{G}_{a}$, and
$\Af^{1}\left(\ZZ\right)$ for the set of integer points in $\Af^{1}$.
As sets, both $\mathbb{G}_{a}\left(\ZZ\right)$ and $\Af^{1}\left(\ZZ\right)$
can be identified with $\ZZ$, and both $\mathbb{G}_{a}$ and
$\Af^{1}$ can be identified with $\RR$. $\mathbb{G}_{a}\left(\ZZ\right)$
also acts on $\Af^{1}$ by translations, and so does any cyclic subgroup
$\left\langle \b\right\rangle =\b\ZZ$ of $\mathbb{G}_{a}\left(\ZZ\right)$.
For $\a\in\Af^{1}\left(\ZZ\right)$ and $\b\in\mathbb{G}_{a}\left(\ZZ\right)$,
the orbit $\left\langle \b\right\rangle \cdot\a$ is simply the arithmetic
progression $\a+\b\ZZ$. Since in $\Af^{1}$ a set is Zariski-dense
if and only if it is infinite, we see that Dirichlet's theorem can
be stated as follows: the set of prime points in an orbit $\left\langle \b\right\rangle \cdot\a$
is Zariski-dense in $\Af^{1}$ if and only if $\gcd\left(\a,\b\right)=1$.

In recent years, several questions pertaining to Zariski-density of
prime points in orbits of the group of integral points of an algebraic
group defined over $\mathbb{Q}$ have been formulated. These questions
continue the longstanding tradition of questions regarding prime values of  algebraic curves. An example for the latter is the Euler conjecture,
which concerns prime values of the quadratic function $x^{2}+1$, or the twin prime
conjecture, which concerns values of the quadratic function $x(x+2)$ with exactly two prime factors.
Concerning the case where the algebraic varieties are orbits of algebraic
groups, Bourgain, Gamburd and Sarnak have formulated some conjectures,
one of which is the following \cite[Conj. 1.4]{BGS_10} (see also
\cite{Sarnak_lecture}):
\begin{conjecture*}[Bourgain, Gamburd, Sarnak]
Let $G\subset GL_{n}\left(\F\right)$ be the group of real points of an algebraically connected, algebraically simply  connected, absolutely almost simple
linear algebraic group defined over $\mathbb{Q}$. Let $\Lambda=\mathcal{O}$ be a subgroup of $G\left(\ZZ\right)=G\cap GL_{n}\left(\ZZ\right)$
which is Zariski-dense in $G$, and $f$ be a non-zero polynomial
 in the coordinate ring $\mathbb{Q}[G]$. Assume that $f$ is not a unit, that it assumes integral values on $\mathcal{O}$, and that it factors into $k$ irreducibles in the coordinate ring $\mathbb{Q}[G]$.
Denote by $r_{0}\left(\mathcal{O},f\right)$ the least $r$ such that
the set of $x\in\mathcal{O}$ for which $f\left(x\right)$
has at most $r$ prime factors, is Zariski-dense in $G$, the Zariski-closure
of $\mathcal{O}$. Assume that for every integer $q\geq2$, there
exists $x\in\mathcal{O}$ such that $\gcd\left(f\left(x\right),q\right)=1$
(such a pair $\left(\mathcal{O},f\right)$ is called \emph{primitive}).
Then $r_{0}\left(\mathcal{O},f\right)=k$.\end{conjecture*}
\begin{rem*}
Note that $\mathbb{Q}[G]$ is indeed a unique factorization domain so that the number of irreducible factors of $f$ is well-defined, see the discussion and references preceding \cite[Conj. 1.4]{BGS_10}. Also, note that if $\left(\mathcal{O},f\right)$ is primitive then in particular $\gcd\left\{ f\left(x\right)\mid x\in\mathcal{O}\right\} =1$.  
\end{rem*}

The Bourgain-Gamburd-Sarnak conjecture has been established in some cases (\cite{Iw},
\cite{F-I}, \cite{H-M}, \cite{Ba}, \cite{G-T2}, \cite{Sa3}),
one of which is the following \cite{NS10}. Let $G=\mbox{SL}_{n}\left(\RR\right)$
and let $\mbox{Mat}_{n}\left(\F\right)$ be the space of $n\times n$
matrices over $\RR$, which we identify with the affine space $\Af^{n^{2}}$.
Consider the action of $G$ on $\mbox{Mat}_{n}$ by left matrix multiplication
$g\cdot x=gx$ for $g\in G$ and $x\in\mbox{Mat}_{n}$. Disregarding
the variety of singular matrices, each $\mbox{SL}_{n}$-orbit is of
the form
\begin{equation}
\vdet_{m}=\left\{ x\in\mbox{Mat}_{n}\left(\RR\right):\det\left(x\right)=m\right\} \label{eq: fixed det variety}
\end{equation}
for $0\neq m\in\F$. Let $\Lambda=\mbox{SL}_{n}\left(\ZZ\right)$, and call an integral matrix prime if all of its entries are prime numbers ($\neq 1$) in $\mathbb{Z}$.
\begin{thm}[\cite{NS10}]
\label{thm: determinant variety}For $n\geq3$ and $m\neq0$, prime
matrices are Zariski-dense in $\vdet_{m}$ if and only if $m\equiv0\left(\mbox{mod }2^{n-1}\right)$.
\end{thm}
If prime matrices are Zariski-dense in $\mathcal{D}_{m}$, then, since
$\mathcal{D}_{m}\left(\ZZ\right)$ is a union of finitely many $\mbox{SL}_{n}\left(\ZZ\right)$-orbits,
there exists at least one $\mbox{SL}_{n}\left(\ZZ\right)$-orbit $\mathcal{O}\subset\mathcal{D}_{m}$
such that $\mathcal{O}$ is Zariski-dense in $\mathcal{D}_{m}$ and
prime matrices are Zariski-dense in $\mathcal{O}$. This means that
when there are no congruence obstructions, $r_{0}\left(\mathcal{O},f\right)=n^{2}$
for $f\left(x_{1,1},\ldots,x_{n,n}\right)=\prod_{i,j=1}^{n}x_{i,j}$.

It is a natural problem to find further infinite families of examples
where the conjecture holds, and this is the goal of the present paper.
We extend the approach of \cite{NS10} to varieties given as the level
set of polynomials with a certain structure, that generalizes the
determinant polynomial. More specifically, we consider hyper-surfaces
of the form $\vgen_{m}=\left\{ x:\inpol\left(x\right)=m\right\} $
where $m\in\ZZ$, and formulate a sufficient condition on $\inpol$
such that prime points are Zariski-dense in $\vgen_{m}$ if and only
if $\vgen_{m}$ contains an odd point (Theorem \ref{thm: General Thm}).
While the condition on $\vgen_{m}$ is clearly necessary, it is interesting
that under certain conditions on the structure of $\inpol$ it is
also sufficient.

The determinant variety $\vdet_{m}$ is one instance in which our
method holds; we proceed to discuss some further examples.

In the first example, we consider the quadratic form on $\RR^{2n+k}$:
\[
Q_{n,k}\left(x_{1},\ldots,x_{n},y_{1},\ldots,y_{n},z_{1},\ldots,z_{k}\right)=\sum_{i=1}^{n}x_{i}y_{i}+\sum_{i=1}^{k}z_{i}^{2}
\]
and the algebraic variety defined as the level set of this form,
\[
\vquad_{m}=\left\{ \left(x,y,z\right)\in\RR^{2n+k}\mid Q_{n,k}\left(x,y,z\right)=m\right\} ,
\]
where $m\neq0$ is an integer. The variety $\vquad_{m}$ is an orbit
of the orthogonal group of the form $Q_{n,k}$ in its action on $\RR^{2n+k}$
by matrix multiplication. Since this group is conjugate to the orthogonal
group $\mbox{SO}_{2n,k}$, for $n\ge 3$ it is a simple (but not simply connected) algebraic group.
The following is a necessary and sufficient condition for Zariski
density of prime vectors in $\vquad_{m}$, namely vectors all of whose entries are prime numbers ($\neq 1$) in $\mathbb{Z}$:

\begin{mythm}

\label{thm: quadratic form variety}

Let $n\geq3$, $k\geq0$ and $m\neq0$ be integers. Prime vectors 
are Zariski-dense in $\vquad_{m}$ if and only if $n+k\equiv m\left(\mbox{mod }2\right)$.

\end{mythm}

In the second example, we consider the variety of $2n\times2n$ anti-symmetric
matrices of fixed Pfaffian $m$:
\[
\vpf_{m}=\left\{ x\in\mbox{Mat}_{2n}(\RR)\mid x^{\mbox{t}}=-x,\pf\left(x\right)=m\right\} .
\]
For $m\neq0$, this variety is an $\mbox{SL}_{2n}$-orbit under the
action: $g\cdot x=gxg^{\mbox{t}}$, $g\in\mbox{SL}_{2n}$. An anti-symmetric matrix all of whose non-diagonal entries are primes numbers $(\neq 1$) in $\mathbb{Z}$ will be called a prime matrix in $\vpf_{m}$. We prove
the following:

\begin{mythm}

\label{thm: Pfaffian variety}For $n\geq2$, prime matrices are Zariski-dense
in $\vpf_{m}$ if and only if $m$ is an odd integer.

\end{mythm}

While the Bourgain-Gamburd-Sarnak conjecture has been formulated for simple groups $G$,
it is of course natural to consider the case where $G$ is semi-simple
as well. Indeed, our third infinite family of examples consists of
orbits of $G=\mbox{SP}_{\l}(\RR)\times\mbox{SL}_{2n}(\RR)$ acting on the space
of $2\l\times2n$ matrices by $\left(g_{1},g_{2}\right)\cdot x=g_{1}xg_{2}^{\mbox{t}}$
for $g_{1}\in\mbox{SP}_{\l}$, $g_{2}\in\mbox{SL}_{2n}$, and $x\in\mbox{Mat}_{2\l\times2n}(\RR)$
with $\l\geq n\geq1$. Let
\begin{equation}
\Omega_{\l}=\left[\begin{matrix}0_{\l} & I_{\l}\\
-I_{\l} & 0_{\l}
\end{matrix}\right].\label{eq: Omega def}
\end{equation}
Let $m\in\ZZ$ and define the variety
\[
\vrec_{m}=\left\{ x\in\mbox{Mat}_{2\l\times2n}(\RR)\mid\pf\left(x^{\mbox{t}}\Omega_{\l}x\right)=m\right\} ,
\]
for which we prove the following:

\begin{mythm}

\label{thm: Rectangular matrices variety}For $n\geq1$, $\l\geq2$,
$\l\geq n$, and $0\neq m\in\ZZ$, prime matrices are Zariski-dense
in $\vrec_{m}\subset\mbox{Mat}_{\,2\l\times2n}\left(\F\right)$ if
and only if $m\equiv0\left(\mbox{mod }2^{2n-1}\right)$.

\end{mythm}

The motivation for this family of varieties comes from pre-homogeneous
vector spaces, to which the families $\left\{ \vdet_{m}\right\} _{m\neq0}$,
$\left\{ \vquad_{m}\right\} _{m\neq0}$ and $\left\{ \vpf_{m}\right\} _{m\neq0}$
belong as well, as explained in Section \ref{sec: rectangular matrices}.

In section \ref{sec: non-homogenous varieties}, using the same technique,
we prove Zariski-density of prime points in some non-homogeneous varieties,
which do not support a group action at all. The main example is the
variety of $n\times n$ matrices of fixed \emph{permanent} $m$:
\[
\vper_{m}=\left\{ x\in\mbox{Mat}_{n}\left(\RR\right)\mid\perm\left(x\right)=m\right\} ,
\]
for which we prove the following.

\begin{mythm}\label{thm: permanent variety}Let $n\geq3$ and $0\neq m\in\ZZ$.
Write $2^{s}-1\leq n<2^{s+1}-1$ for a unique integer $s\geq2$ (i.e.
$s=\left\lfloor \log_{2}\left(n+1\right)\right\rfloor $). Then prime
matrices are Zariski-dense in $\vper_{m}$ if and only if
\[
m\equiv\begin{cases}
2^{n-s}\left(\mbox{mod }2^{n-s+1}\right) & \mbox{if }n=2^{s}-1\\
0\left(\mbox{mod }2^{n-s}\right) & \mbox{if }2^{s}-1<n<2^{s+1}-1.
\end{cases}
\]

\end{mythm}

Note that in this case, the set of $m\in\ZZ$ where ${\cal P}_{m}$
is Zariski dense, depends on whether $n+1$ is a power of $2$ or
not!

\begin{rem}
A key ingredient of our method is that, as was noted in \cite{NS10},
establishing Zariski-density of prime points in such varieties $\vgen_{m}$
can be reduced to solving a non-homogeneous linear Diophantine equation
\[
\a_{1}\xi_{1}+\ldots+\a_{n}\xi_{n}+\a_{n+1}=0
\]
in primes. This is possible by a theorem of Vaughan \cite{Vaughan}
(based on Vinogradov's Three-Prime Theorem), when the integers $\a_{1},\ldots,\a_{n+1}$
satisfy certain congruence conditions and $n\geq3$. It will be discussed
extensively in Section \ref{sec: Prime solutions - Vaughan}.
\end{rem}

\paragraph*{Connection to  the Bourgain-Gamburd-Sarnak conjecture.}

We remark that Theorems %\ref{thm: determinant variety}, 
\ref{thm: quadratic form variety},
\ref{thm: Pfaffian variety}, and \ref{thm: Rectangular matrices variety}
are indeed new instances of the conjecture stated above. The common setting 
of these examples (generalizing that of Theorem \ref{thm: determinant variety}) 
 is that of a semi-simple linear algebraic
group $G$ acting on a real finite dimensional linear space $V\left(\RR\right)$, realized as
a space of matrices $\mbox{Mat}_{n_{1}\times n_{2}}\left(\RR\right)$.
There exists a polynomial $\inpol$ on $V\cong\RR^{\dim V}$ which
is invariant under the action of $G$ (namely $\Delta\left(g\cdot x\right)=\Delta\left(x\right)$)
such that the level sets $\vgen_{m}=\left\{ x\in V:\inpol\left(x\right)=m\right\} $
for $m\neq0$ are orbits of $G$. (The level set for $m=0$ consists
of several $G$-orbits and will not be part of our discussion). When
restricting to integral points, $G\left(\ZZ\right)$
acts on $V\left(\ZZ\right)$, and $V\left(\ZZ\right)=\bigcup_{m\in\ZZ}\vgen_{m}\left(\ZZ\right)$,
since $\inpol$ has integer coefficients. Moreover, every $\vgen_{m}\left(\ZZ\right)$
for $0\neq m\in\ZZ$ is a finite union of $G\left(\ZZ\right)$-orbits;
hence, since $\vgen_{m}\left(\ZZ\right)$ is Zariski-dense in $\vgen_{m}$,
there exists at least one such $G\left(\ZZ\right)$-orbit ${\cal O}$ which 
is also Zariski-dense
in $\vgen_{m}$.

Theorems \ref{thm: determinant variety},
\ref{thm: quadratic form variety}, \ref{thm: Pfaffian variety},
and \ref{thm: Rectangular matrices variety} assert that when there are
no congruence obstructions, prime points are Zariski dense in $\vgen_{m}$
(the Zariski closure of ${\cal O}$); this means that for the polynomial
$f\left(x_{1},\ldots,x_{\dim V}\right)=\prod_{i=1}^{\dim V}x_{i}$,
 $r_{0}\left({\cal O},f\right)=\dim V$. 
 
 The Bourgain-Gamburd-Sarnak conjecture is formulated specifically for a simply connected simple algebraic group $\tilde{G}\subset GL_n$ defined over $\mathbb{Q}$, 
 %and a principal homogeneous space (namely with trivial stability group), 
 with the orbit being a Zariski-dense subgroup $\Lambda$ of $\tilde{G}(\mathbb{Z})$. 
 To put this in perspective, note that in \cite[\S 2.2]{BGS_10}, the example of the double-cover adjoint epimomorphism  $\phi: \tilde{G}=SL_2(\mathbb{R})\to SO(F)=G$ is considered, where $F$ is the three variables form $F(x,y,z)=xz-y^2$. It is shown there that the conjecture fails for the adjoint group $SO(F)$ and the primitive polynomial $f(x,y,z)=x_{1,1}-1$ in $\mathbb{Q}[G]$. 

When $G$ is not simply connected, the covering map $ \tilde{G}\to G $, $\tilde{g}\to g$ is not injective, and to compare the conjecture to our set-up we follow the discussion in  \cite[\S 2.2]{BGS_10}. Given the $G$-orbit $G\cdot v_0=\vgen \subset \RR^d\cong V\cong G/H$ with $v_0\in \mathbb{Z}^d$ and $H=\text{St}_G(v_0)$, consider the map $\phi : \tilde{G}\to \vgen$ given by $\phi(\tilde{g})=g \cdot v_0$, so that $\mathcal{O}=\phi(\tilde{G}(\mathbb{Z}))= G\cdot v_0 \subset \vgen(\mathbb{Z})$. 
  There is an injective ring homomorphism between the coordinate rings $\phi^\ast : \mathbb{Q}[G/H]=\mathbb{Q}[\vgen]\to  \mathbb{Q}[\tilde{G}]$, given by $\phi^\ast(f)=f\circ \phi$. 
 %Similarly, when the orbit $\mathcal{O}=G\cdot v_0\cong G/H$ in question is not principal, namely the stability group $H$ is not trivial, again there is an injective ring homomorphism $\phi^\ast  :\mathbb{Q}[G/H]\to \mathbb{Q}[\tilde{G}]$. 
 Given a polynomial $f$ in $\mathbb{Q}[\vgen]=\mathbb{Q}[G/H]$ which is primitive on $\mathcal{O}$, clearly $\phi^\ast(f)\in \mathbb{Q}[\tilde{G}]$ is primitive on $\tilde{G}(\mathbb{Z})=\Lambda$, and then the conjecture asserts that %setting $\Lambda =\tilde{G}(\mathbb{Z})$,  
 $r_0\left(\Lambda, \phi^\ast(f)\right)$ is equal to the number of irreducible factors in  the factorization of $\phi^\ast(f)$ to irreducibles in the unique factorization domain $\mathbb{Q}[\tilde{G}]$.

 Let us begin by noting that $r_0\left(\Lambda, \phi^\ast(f)\right)$ is equal to $r_0(\mathcal{O},f)$, arguing as follows. First, the set of points $x\in \mathcal{O}\subset \vgen(\mathbb{Z})\subset V\cong G/H$ where $f(x)$ has strictly less than  $r_0(\mathcal{O},f)$ prime factors has a non-trivial polynomial vanishing on it, and so the same is true of its inverse image in $\tilde{G}(\mathbb{Z})=\Lambda$, namely the set of $\tilde{x}\in \Lambda$ where $\phi^\ast(f)(\tilde{x})$ has less than $r_0(\mathcal{O},f)$ prime factors. Hence 
 $r_0\left(\Lambda, \phi^\ast(f)\right)\ge r_0(\mathcal{O},f)$. Second, to see that the set $Z^0$ of $\tilde{x}\in \Lambda$ where $\phi^\ast(f)(\tilde{x})$ has at most $ r_0(\mathcal{O},f)$ prime factors is Zariski dense in $\tilde{G}$, assume for constradiction that it is not, and let $Z\subsetneq \tilde{G}$ denote its Zariski closure. Clearly, if $\tilde{H}(\mathbb{Z})=\tilde{H}\cap \tilde{G}(\mathbb{Z})$, then $Z^0 \tilde{h}=Z^0$ for every $\tilde{h}\in \tilde{H}(\mathbb{Z})$, and hence 
  $Z\tilde{h}=Z$ for every $\tilde{h}\in \tilde{H}(\mathbb{Z})$. It follows that 
  $Z\tilde{h}=Z$ for every $\tilde{h}\in \tilde{H}$, since $\tilde{H}$ is the Zariski closure of $\tilde{H}(\mathbb{Z})$. This is a consequence of the Borel density theorem, since $\tilde{H}$ is a (semi)simple algebraic group defined over $\mathbb{Q}$ in the examples under consideration, and $\tilde{H}(\mathbb{Z})\subset \tilde{H}$ is a lattice subgroup. As a result, $\phi(Z)\subset \vgen\cong G/H$ is a proper Zariski closed subset containing the set of all $x\in \mathcal{O}$ 
  having the property that $f(x)$ is the product at most $r_0(\mathcal{O},f)$ prime factors. Since the latter set is Zariski dense by definition of $r_0(\mathcal{O},f)$, we have arrived at a contradiction, and as a result $r_0(\mathcal{O}, f)=r_0\left(\Lambda, \phi^\ast(f)\right)$.

%  {\bf we will show that it intersects every non-empty open set. Indeed $\tilde{G}$ is an irreducible variety, and if $\tilde{U}\subset \tilde{G}$ is non-empty and open then it is irreducible and dense, and its image $U=\phi(\tilde{U})\subset \vgen \cong G/H$ contains a non-empty Zariski-open set $W$ in its Zariski closure (see
%\cite[Theorem. 1.9.5]{SpringerLinearAlgGrps}). But the Zariski closure $\bar{U}$ of $U$ clearly coincides with $\vgen\cong G/H$, otherwise its inverse image $\phi^{-1}(\bar{U})$ will be a proper Zariski closed subset containing $\tilde{U}$, which is dense in $\tilde{G}$. Hence $W$ is non-empty and open in $\vgen \cong G/H$, and so it intersects the Zariski-dense set of points $x \in \mathcal{O}$ where $f(x)$ has at most $r_0(\mathcal{O},f)$ prime factors. For a point $x$ in the intersection, if $\tilde{x}\in \tilde{G}(\mathbb{Z})$ satisfies $\phi(\tilde{x})=x$, then $\tilde{x}$ belongs to  $\phi^{-1}(W)\subset \phi^{-1}(\phi(\tilde{U}))$....................................} 
%
%

 Theorems \ref{thm: quadratic form variety},
\ref{thm: Pfaffian variety}, and \ref{thm: Rectangular matrices variety} establish that $r_0(\mathcal{O}, f)$ is equal to the number of irreducible factors of $f$ in $\mathbb{Q}[G/H]=\mathbb{Q}[\vgen]$, namely $d=\dim V$. 
Thus the verification of the conjecture for 
the pair $\left(\Lambda, \phi^\ast(f)\right)$ will be complete upon showing that the number of irreducible factors of $\phi^\ast(f)$ in $\mathbb{Q}[\tilde{G}]$ is equal to $r_0(\mathcal{O},f)$, and no more. Let $\phi^\ast(f)=h_1\cdots h_s$ be the decomposition into non-trivial irreducibles in the unique factorization domain $\mathbb{Q}[\tilde{G}]$, and assume for contradiction $s > d= r_0(\mathcal{O},f)$.  
The group of real points $\tilde{G}(\mathbb{R})$ is contained in $M_n(\RR)$ for some $n$, and coincides with the set of common zeros of an ideal $\mathcal{J}$ in the polynomial ring $\mathbb{Q}[\{t_{i,j}\}_{i,j=1}^n]$. 
%Furthermore, in the cases under considerations, the ideal $\mathcal{J}$ has a finite set of generators which are polynomials with integral coefficients. 
We can represent each element $h_i$ in the ring $\mathbb{Q}[\tilde{G}]$ in the form $\frac{a_i}{b_i} h_i^\prime + \mathfrak{j}_i$, where $h_i^\prime\in \mathbb{Q}[\{t_{i,j}\}_{i,j=1}^n]$ is a polynomial with integral coefficients whose greatest common divisor is $1$ (for definiteness), $a_i,b_i\in \mathbb{Z}\setminus\{0\}$, and $\mathfrak{j}_i\in \mathcal{J}$.  Then $\phi^\ast(f)=\frac{a}{b}h_1^\prime\cdots h_s^\prime +\mathfrak{j}$
%{\bf If we can assume that there exists a Zariski-dense subset of $V$ such that each $h_i$ assumes integral values on its inverse image}
and for $\tilde{x}\in \Lambda$, 
% setting $d=\dim V=r_0(\mathcal{O},f)$, $\tilde{x}\in \tilde{G}$, and 
and $\phi(\tilde{x})=x=(x_1,\dots,x_d)\in \mathcal{O}\subset \vgen(\mathbb{Z})$, we have the identity (since $\mathfrak{j}(\tilde{x})=0$)  
$$\phi^\ast(f)(\tilde{x})=\frac{a}{b} h_1^\prime(\tilde{x})\cdots h^\prime_s(\tilde{x})=f(\phi(\tilde{x}))=f(x)=x_1\cdots x_d $$
%We can then write $\phi^\ast(f)=q h^\prime_1\cdots h^\prime_t$ where $q\in \mathbb{Q}$ and $h_i^\prime$ are non-constant polynomials with integer coefficients, such that the greatest common divisor of the non-zero coefficients of each $h_i^\prime$ is $1$.
For $\tilde{x}\in \tilde{G}(\mathbb{Z})$, each $h^\prime_i(\tilde{x})$ is an integer, and so the previous identity represents the integer $f(x)$ as a product of $d$ integers, and as a product of $s> d$ integers and the rational number $\frac{a}{b}$, where we assume $(a,b)=1$. 
%Furthermore  $\tilde{G}(\mathbb{Z})$ leaves $\mathcal{O}$ invariant and has finitely many orbits in $\mathcal{O}$, so $f\circ\phi(\tilde{x})=f(x)$ is indeed integral for $\tilde{x}\in \tilde{G}(\mathbb{Z})$. Indeed, in the cases under consideration $\phi$ is given by polynomials with integral coefficients.  

 Consider the set $\tilde{x}\in \tilde{G}(\mathbb{Z})=\Lambda$ where the integer $f(\phi(\tilde{x}))$ is a product of {\it exactly} $d=r_0(\mathcal{O},f)$ prime factors. This set is clearly Zariski dense by definition of $r_0(\Lambda, \phi^\ast(f))$ and the fact that it is equal to $r_0(\mathcal{O}, f)$. However, $f(\phi(\tilde{x}))$ is also a product of $s> d$ integer factors $h^\prime_i(\tilde{x})$ and $\frac{a}{b}$.  For this to happen $b$ must cancel against some of the factors dividing $h_i^\prime(\tilde{x})$, and in addition possibly some of the factors $h_i^\prime(\tilde{x})$ are equal to $\pm 1$. It follows that the set in question is contained in the union of the zero sets of the polynomials $h^\prime_i\pm 1$ and $h_i^\prime\pm c$, where $c$ ranges over all the factors of $b$, and $1\le i\le 
 s$.  The latter condition is a consequence of the fact that $s > d$. This last set is not Zariski dense in $\tilde{G}(\mathbb{\RR})$, because if it were, one of the polynomials $h_i^\prime\pm c, h_i^\prime\pm 1$ would vanish identically on $\tilde{G}(\mathbb{R})$, but we have assume that $h_i=\frac{a_i}{b_i}h_i^\prime+\mathfrak{j}_i$ is a non-trivial irreducible element in the ring $\mathbb{Q}[\tilde{G}]$, namely not a constant and not a unit.
 % and that the gcd of the coefficients of $h_i^\prime$ is $1$. 
 Therefore we have arrived at a contradiction, and we can conclude that the conjecture is verified for the pair $(\tilde{G}(\mathbb{Z}),\phi^\ast(f))$. In fact, the arguments above verify the conjecture for the polynomial $\phi^\ast(f)$ and any Zariski dense subgroup $\Lambda\subset \tilde{G}(\mathbb{Z})$ which is transitive on $\mathcal{O}$, and  satisfies also that $\Lambda\cap \tilde{H}(\mathbb{Z})$ is Zariski dense in $\tilde{H}$.

\section{The method of proof: prime solutions to linear equations\label{sec: Prime solutions - Vaughan}}

The varieties that we consider in the present paper are of the form
\begin{equation}
\vgen_{m}=\left\{ x:\inpol\left(x\right)=m\right\} \label{eq: general form of variety}
\end{equation}
with $0\neq m\in\ZZ$,
\[
x=\left(\xi,y,z\right)=\left(\left(\xi_{1},\ldots,\xi_{n}\right),\left(y_{1},\ldots,y_{N}\right),\left(z_{1},\ldots,z_{k}\right)\right),
\]
and
\begin{equation}
\inpol\left(x\right)=F_{1}\left(y\right)\xi_{1}+\ldots+F_{n}\left(y\right)\xi_{n}+G\left(z\right),\label{eq: def of pol Delta}
\end{equation}
where $G$ and $F_{i}$ for all $i=1,\ldots,n$ are polynomials with
integer coefficients, and each $F_i$, $i=1,\ldots,n$ is not the zero polynomial.

For example, the determinant of a matrix $x\in\mbox{Mat}_{n}$ can
be expanded along the $i$-th row, and is therefore ``a linear combination''
of the variables $\left(\xi_{1},\ldots,\xi_{n}\right)=\left(x_{i,1},\ldots,x_{i,n}\right)$,
where the ``coefficients'' are polynomials in the remaining entries
of $x$. As for the remaining examples, we shall verify later on that
these varieties indeed share this structure.

When the variables $y$ and $z$ assume fixed integer values, the equation
$\inpol\left(x\right)=m$ becomes a non-homogeneous linear Diophantine
equation in the variables $\left\{ \xi_{i}\right\} _{i=1}^{n}$. Under
certain necessary congruence conditions on the coefficients, such
equations can be solved in primes:
\begin{thm}[Vaughan, \cite{Vaughan}]
\label{thm: Vaughan}Let $\a_{1},\ldots,\a_{n},m\in\ZZ\setminus\left\{ 0\right\} $
where $n\geq3$, and consider the equation:
\[
\a_{1}\xi_{1}+\ldots+\a_{n}\xi_{n}=m.
\]
Let $T$ be a large positive integer, and let:
\[
\solspace_{\mbox{prime}\leq T}=\left\{ \left(p_{1},\ldots,p_{n}\right)\left|\begin{split} & p_{i}\mbox{ are prime}\\
 & \left|p_{i}\right|\leq T\\
 & \a_{1}p_{1}+\ldots+\a_{n}p_{n}=m
\end{split}
\right.\right\} .
\]
Then for every fixed large $C>0$
\begin{equation}
\left|\solspace_{\mbox{prime}\leq T}\right|\geq\mathfrak{S}\cdot\frac{T^{n-1}}{\left(\log T\right)^{n}}+O_{C}\left(\frac{T^{n-1}}{\left(\log T\right)^{C}}\right)\label{eq: vaughan: H_prime}
\end{equation}
where $\mathfrak{S}>0$ if and only if for all $i=1,\ldots,n$
\begin{equation}
\gcd\left(\a_{1},\ldots,\a_{n}\right)=\gcd\left(\left\{ \a_{1},\ldots,\a_{n},m\right\} \setminus\left\{ \a_{i}\right\} \right)\label{eq: vaughan: gcd's are equal}
\end{equation}
 and
\begin{equation}
\a_{1}+\ldots+\a_{n}-m\equiv0\left(\mbox{mod }2\cdot\gcd\left(\a_{1},\ldots,\a_{n},m\right)\right).\label{eq: vaughan: sum is zero mod 2d}
\end{equation}

\end{thm}
Theorem \ref{thm: Vaughan} is actually a variation on Vinogradov's
three-prime theorem (\cite{Vin_37}; see also \cite{Vaughan} and
\cite{NathansonAdditive}) and is stated as an exercise in \cite{Vaughan};
for the details of the proof, see \cite{H12}. The following is a
consequence of Theorem \ref{thm: Vaughan}.
\begin{cor}
\label{cor: prime solutions to affine equation}Let $\solspace$ be
the affine space of solutions to the non-homogeneous linear equation:
\[
\a_{1}\xi_{1}+\ldots+\a_{n}\xi_{n}=m
\]
where $\a_{1},\ldots,\a_{n},m\in\ZZ\setminus\left\{ 0\right\} $ and
$n\geq3$. Let $\solspace_{\mathrm{prime}}\subset\solspace$ be the
set of prime points in $\solspace$, namely the set of prime vector solutions
to the given equation. Assume the integers $\a_{1},\ldots,\a_{n},m$
satisfy conditions \ref{eq: vaughan: gcd's are equal} and \ref{eq: vaughan: sum is zero mod 2d}
stated in Theorem \ref{thm: Vaughan}. Then $\solspace_{\mathrm{prime}}$
is Zariski-dense in $\solspace$.
\end{cor}

\paragraph{Notation: $Z\left(h\right)$.}

For a polynomial $h\left(\xi_{1},...,\xi_{n}\right)$, we let $Z\left(h\right)$
denote the zero-set of $h$, namely the set $Z\left(h\right)=\left\{ \left(\xi_{1},...,\xi_{n}\right)\mid h\left(\xi_{1},...,\xi_{n}\right)=0\right\} $.

\begin{proof}
$\mathcal{H}$ is a translation of a linear space of dimension $n-1$
and therefore $\solspace\cong\Af^{n-1}$. In particular, $\solspace$
is irreducible. Assume $\solspace_{\mathrm{prime}}$ is not Zariski-dense
in $\solspace$. Since $\solspace$ is irreducible, this means that
there exists a polynomial $h\left(\xi_{1},...,\xi_{n}\right)$ such
that $h\left(\xi_{1},...,\xi_{n}\right)=0$ for all $\left(\xi_{1},...,\xi_{n}\right)\in\solspace_{\mathrm{prime}}$,
and $h\not\equiv0$ on $\solspace$. We may assume $h$ to be irreducible.
$\solspace$ is an irreducible affine variety of dimension $n-1$,
hence any proper closed hyper-surface inside it is of dimension $n-2$.
In particular, the zero set $Z\left(h\right)$ is of dimension $n-2$.
 It follows (\cite[Lemma 1]{Lang_Weil_54}) that the number of integer
points inside
\begin{equation}
\solspace\cap\left\{ \left(\xi_{1},...,\xi_{n}\right):h\left(\xi_{1},...,\xi_{n}\right)=0\right\} \cap\left\{ \left(\xi_{1},...,\xi_{n}\right):|\xi_{i}|\leq T\mbox{ for all }i\right\} \label{eq: H intersection Z(h)}
\end{equation}
is bounded by a constant times $T^{n-2}$, namely it is $O\left(T^{n-2}\right)$.
As we assume $\solspace_{\mathrm{prime}\leq T}$ is contained in the
set \ref{eq: H intersection Z(h)}, this contradicts \ref{eq: vaughan: H_prime}
with $\mathfrak{S}>0$.
\end{proof}
Theorem \ref{thm: Vaughan} and specifically Corollary \ref{cor: prime solutions to affine equation}
are the key ingredients of the proof of Theorem \ref{thm: determinant variety},
as well as of the examples that we consider in this paper. We shall
use the technique presented in \cite{NS10} as follows.
\begin{thm}
\label{thm: Prime solutions variety}Let $m\in\ZZ$, $m \neq 0$ and consider a
variety $\vgen_{m}$ of the form \ref{eq: general form of variety},
with $F_{1},\ldots,F_{n}$ and $G$ as in \ref{eq: def of pol Delta}.
Assume that $G\left(z\right)$ is not identically equal to $m$, and
that there exists a Zariski-dense subset $\G\subset\RR^{N+k}$
whose elements are prime vectors $\left(y,z\right)\in\ZZ^{N+k}$
that satisfy
\[
F_{1}\left(y\right)+\ldots+F_{n}\left(y\right)+G\left(z\right)-m\equiv0\left(\mbox{mod }2\cdot\gcd\left(F_{1}\left(y\right),\ldots,F_{n}\left(y\right),G\left(z\right)-m\right)\right),
\]
\[
\gcd\left(\left\{ F_{1}\left(y\right),\ldots,F_{n}\left(y\right),G\left(z\right)-m\right\} \setminus\left\{ F_{j}\left(y\right)\right\} \right)=\gcd\left(F_{1}\left(y\right),\ldots,F_{n}\left(y\right)\right)
\]
for all $j=1,\ldots,n$. Then  prime points are Zariski-dense in
$\vgen_{m}$.\end{thm}

\begin{rem}
Except for the variety $\vquad_{m}$ in Theorem \ref{thm: quadratic form variety},
all the examples considered in this paper have  $G\left(z\right)$
which identically equals zero.
\end{rem}
Theorem \ref{thm: Prime solutions variety} is a consequence of Corollary
\ref{cor: prime solutions to affine equation}, along with the following
two observations --- the first is that the varieties $\vgen_{m}$
are irreducible.
\begin{lem}
\label{lem: Delta is irreducible}The polynomial
\[
\inpol\left(\xi,y,z\right)=\sum_{i=1}^{n}F_{i}\left(y\right)\xi_{i}+G\left(z\right),
\]
where $G\left(z\right)$ is not the zero polynomial, is irreducible.
\end{lem}
\begin{proof}
Suppose $\inpol=f\cdot h$. Assume without the loss of generality
that $\deg_{\xi_{1}}f\geq1$ (namely $\deg_{\xi_{1}}f=1$), and therefore
$\deg_{\xi_{1}}h=0$. Since monomials of the form $\xi_{1}\xi_{i}$
do not appear in $\inpol$, it follows that for every $i=i,\ldots,n$
$\deg_{\xi_{i}}f=1$ and $\deg_{\xi_{i}}h=0$. Thus,
\begin{align*}
f & =\sum_{i=1}^{n}f_{i}\left(y,z\right)\xi_{i}+g\left(y,z\right)\\
h & =h\left(y,z\right).
\end{align*}
It follows that
\[
G\left(z\right)=g\left(y,z\right)\cdot h\left(y,z\right),
\]
which implies
\begin{align*}
g\left(y,z\right) & =g\left(z\right)\\
h\left(y,z\right) & =h\left(z\right).
\end{align*}
Hence
\[
\inpol=\left(\sum_{i=1}^{n}f_{i}\left(y,z\right)\xi_{i}+g\left(z\right)\right)\cdot\left(h\left(z\right)\right).
\]
Then
\[
F_{i}\left(y\right)=f_{i}\left(y,z\right)\cdot h\left(z\right),
\]
namely $f_{i}\left(y,z\right)=F_{i}\left(y\right)$ and $h\left(z\right)$
is a scalar.
\end{proof}

We note that the homogeneous varieties we consider each constitutes an orbit of a connected algebraic group, and hence are clearly irreducible. But for the non-homogeneous varieties we consider the previous argument is necessary.

Let us now formulate the context in which Vaughan's criterion will be applied. In what follows, we use the term "algebraic variety" as an abbreviation for the term "the set of real points of an algebraic variety defined over $\mathbb{R}$", which describes all the varieties we will consider in the present paper.  

\global\long\def\denset{{\cal T}}
\global\long\def\var{X}

\begin{lem}
\label{lem: for each set there is set- dense}Let $\Xi$ and $Y$
be algebraic varieties, and let $\var\subset\Xi\times Y$ be an
irreducible subvariety. Let $A$ be a Zariski-dense subset of $Y$
such that for every $y\in A$ the fiber $\var_{y}:=\left\{ \xi:\left(\xi,y\right)\in\var\right\} $
is non-empty. Assume that for every $y\in A$ there exists a Zariski-dense
subset $B_{y}$ in $\var_{y}$. Then, the set
\[
\denset=\left\{ \left(\xi,y\right):y\in A,\xi\in B_{y}\right\}
\]
is Zariski-dense in $\var$.
\end{lem}
%Note that $\denset$ is typically not a product set.
\begin{proof}
We let $U\subset\var$ be a non-empty open set, and show that $U$
contains a point from $\denset$. Let $\phi:\var\to Y$ be the natural
projection, which is defined over $\mathbb{R}$ and has a Zariski-dense image, by assumption.

%Clearly, If a subset $V$ of $Y$  is open and dense then $\phi^{-1}(V)$ is open and dense in $X$, and hence $\phi$ is a dominant morphism. By Chevalley's theorem 
%(see) 

Since $\var(\CC)$ is irreducible and $U(\CC)$ is open and non-empty,
$U(\CC)$ is irreducible and dense in $\var(\CC)$. Therefore, $\phi\left(U(\CC)\right)$
contains a non-empty open set $W(\CC)$ (see e.g. 
\cite[Theorem. 1.9.5]{SpringerLinearAlgGrps}) of $\overline{\phi\left(U(\CC)\right)}=Y(\CC)$.
%The fiber $X_{y}$ above every $y\in A$ is non-empty, namely $A\subset\phi\left(\var\right)$.
%Thus
%\[
%Y=\overline{A}\subseteq\overline{\phi\left(\var\right)}=\overline{\phi\left(U\right)},
%\]
%that is $\overline{\phi(U)}=Y$. Then $W\subseteq\phi\left(U\right)$
%is open in $Y$.
Then $W=W(\CC)\cap Y$ is a non-empty open set of $Y$, since the set of real points $Y$ of the variety $Y(\CC)$ which is defined over $\RR$ are Zariski dense in $Y(\CC)$ and thus intersect every non-empty open set (see e.g.  \cite[Chap. AG, Cor. 13.3]{Bo}). 
Since $A$ is dense in $Y$ it intersects $W$; let
\[
y\in A\cap W\subseteq A\cap\phi\left(U\right).
\]
Then $\phi^{-1}\left(y\right)\cap U$ is non-empty, and clearly open
in $\phi^{-1}\left(y\right)$. By projecting to $\Xi$, we may identify
$\phi^{-1}\left(y\right)$ with $X_{y}$, and $\phi^{-1}\left(y\right)\cap U$
with an open subset of $X_{y}$; this open subset intersects $B_{y}$,
which is assumed to be Zariski-dense in $X_{y}$. Let $\xi$ be a
point in this intersection; then $\left(\xi,y\right)$ is contained
in $U\cap\denset$.
\end{proof}
The following special case of Lemma \ref{lem: for each set there is set- dense},
where all the fibers $X_{y}$ coincide,
will be used later on.
\begin{example}
\label{exa: for each set there is a set - dense}Let $A\subset\Af^{m}$
be a Zariski-dense subset such that that for every $a\in A$ there
exists a subset $B_{a}\subset\Af^{n}$ which is Zariski-dense in $\mbox{\ensuremath{\Af}}^{n}$.
Then the set
\[
\denset=\left\{ \left(a,b_{a}\right)\mid a\in A,b_{a}\in B_{a}\right\}
\]
is Zariski-dense in $\Af^{m+n}$.
\end{example}
We conclude this section with a proof of
Theorem \ref{thm: Prime solutions variety}.

\begin{proof}
For every $\left(y,z\right)$, let $\solspace\left(y,z\right)$ denote
the space of solutions to the non-homogeneous linear equation:
\[
F_{1}\left(y\right)\xi_{1}+\ldots+F_{n}\left(y\right)\xi_{n}+G\left(z\right)=m,
\]
and let $\solspace_{\mbox{prime}}\left(y,z\right)\subset\solspace\left(y,z\right)$
denote the subset of prime solutions to this equation. Define
\[
\tilde{\G}=\left\{ \left(y,z\right)\in\G\mid G\left(z\right)\neq m\mbox{ and }\forall j:\, F_{j}\left(y\right)\neq0\right\} \subset\G,
\]
and note that $\tilde{\G}$ is Zariski-dense in $\RR^{N+k}$,
since it is obtained from $\G$ by removing its intersection with
the two Zariski-closed subsets defined by $\left\{ G\left(z\right)=m\right\} $
and $\cup_{j=1}^{n}\left\{ F_{j}\left(y\right)=0\right\} $, which are proper subsets since we assume each $F_{j}\left(y\right)$ is not the zero polynomial and $G(z)$ is not the constant $m$. Moreover, by Corollary \ref{cor: prime solutions to affine equation}, for every $\left(y,z\right)\in\tilde{\G}$, the fiber $\solspace\left(y,z\right)$
is non-empty,  and 
$\solspace_{\mbox{prime}}\left(y,z\right)$ is Zariski-dense in $\solspace\left(y,z\right)$.  

In the notations of Lemma \ref{lem: for each set there is set- dense},
take $\Xi=\Af^{n}$, $Y=\Af^{N+k}$ and
$X=\vgen_{m}$; note that $\vgen_{m}\subset\Af^{n}\times\Af^{N+k}$
is irreducible by Lemma \ref{lem: Delta is irreducible}. The proof
is concluded by Lemma \ref{lem: for each set there is set- dense},
when taking $A=\tilde{\G}$, $\var_{y}=\solspace\left(y,z\right)$
and $B_{y}=\solspace_{\mbox{prime}}\left(y,z\right)$.
\end{proof}

In order to apply Theorem \ref{thm: Prime solutions variety}
for establishing Zariski-density of prime points in varieties of the
form \ref{eq: general form of variety}, one must establish the existence
of a set $\G$ of prime points $\left(y,z\right)$ on which the polynomials
$F_{i}\left(y\right)$ satisfy the congruence conditions defined in
Theorem \ref{thm: Prime solutions variety} with respect to $m$ and
$G\left(z\right)$. This is the topic of Section \ref{sec: Related polynomials}.

\section{Variety defined by a quadratic form\label{sec: Quadratic forms}}

As mentioned above, we shall consider varieties of the form \ref{eq: general form of variety},
where the polynomials $F_{i}\left(y\right)$ that appear as coefficients
in the form \ref{eq: general form of variety} satisfy some general
conditions (to be formulated in Theorem \ref{thm: General Thm}).
However, in the case of the variety $\vquad_{m}$ considered in Theorem
\ref{thm: quadratic form variety}, the polynomials $F_{i}\left(y\right)$
are quite simple, and so we begin by analyzing this example explicitly.
As we shall see later on, this example already demonstrates the main
ideas of the general case.

Recall from Section \ref{sec: Introduction and examples} the following
quadratic form on $\RR^{2n+k}$:
\[
Q{}_{n,k}\left(\xi_{1},\ldots,\xi_{n},y_{1},\ldots,y_{n},z_{1},\ldots,z_{k}\right)=\sum_{i=1}^{n}\xi_{i}y_{i}+\sum_{i=1}^{k}z_{i}^{2},
\]
and the variety:
\[
\vquad_{m}=\left\{ \left(\xi,y,z\right)\in\RR^{2n+k}\mid Q_{n,k}\left(\xi,y,z\right)=m\right\} .
\]

\def\thethmrpt{\ref{thm: quadratic form variety}}
\begin{thmrpt}

Let $n\geq3$, $k\geq0$ and $m\neq0$ be integers. Prime matrices
are Zariski-dense in $\vquad_{m}$ if and only if $n+k\equiv m\left(\mbox{mod }2\right)$.

\end{thmrpt}

\begin{proof}[Proof of Theorem \ref{thm: quadratic form variety}.]
The condition $n+k\equiv m\left(\mbox{mod }2\right)$ is necessary:
if prime points are Zariski-dense in $\vquad_{m}$, then there exists
an odd point in $\vquad_{m}$, namely there exist odd integers $\xi_{i},y_{i},z_{i}$
which satisfy the equation $\sum_{i=1}^{n}\xi_{i}y_{i}+\sum_{i=1}^{k}z_{i}^{2}=m$.
Take this equation modulo $2$ to obtain $n+k\equiv m\left(\mbox{mod }2\right)$.

For sufficiency, we apply Theorem \ref{thm: Prime solutions variety}.
The variety $\vquad_{m}$ is of the form \ref{eq: general form of variety},
with $N=n$,
\begin{eqnarray*}
F_{1}\left(y_{1},\ldots,y_{N}\right) & = & y_{1}\\
 & \vdots\\
F_{n}\left(y_{1},\ldots,y_{N}\right) & = & y_{n}
\end{eqnarray*}
and
\[
G\left(z_{1},\ldots,z_{k}\right)=\sum_{i=1}^{k}z_{i}^{2}.
\]
Take $\G_{\mbox{quad}}\subset\RR^{n+k}$ to be the set of integer
points $\left(y_{1},\ldots,y_{n},z_{1},\ldots,z_{k}\right)$ such
that $y_{3},\ldots,y_{n},z_{1},\ldots,z_{k}$ are any odd primes satisfying
$m\neq\sum_{i=1}^{k}z_{i}^{2}$, and $y_{1},y_{2}$ are distinct odd
primes that are co-prime to $m-\sum_{i=1}^{k}z_{i}^{2}$. By Lemma
\ref{lem: for each set there is set- dense} (the case of Example
\ref{exa: for each set there is a set - dense}), $\G_{\mbox{quad}}$
is Zariski-dense in $\RR^{n+k}$. For every $\left(y_{1},\ldots,y_{n},z_{1},\ldots,z_{k}\right)\in\G_{\mbox{quad}}$
it holds that
\[
\gcd\left(y_{1},\ldots,y_{n},m-\sum_{i=1}^{k}z_{i}^{2}\right)=1,
\]
\[
\gcd\left(y_{1},\ldots,y_{n}\right)=1,
\]
\[
\gcd\left(\left\{ y_{1},\ldots,y_{n},m-\sum_{i=1}^{k}z_{i}^{2}\right\} \setminus\left\{ y_{j}\right\} \right)=1
\]
for all $j=1,\ldots,n$ and
\[
y_{1}+\ldots+y_{n}+m-\sum_{i=1}^{k}z_{i}^{2}=0\left(\mbox{mod }2\right)
\]
because of the condition on $m,n,k$. By Theorem \ref{thm: Prime solutions variety},
we conclude that prime points are Zariski-dense in $\vquad_{m}$.\end{proof}
\begin{rem}
It would have been simpler to take $\G_{\mbox{quad}}\subset\RR^{n+k}$
to be the set of integer points $\left(y_{1},\ldots,y_{n},z_{1},\ldots,z_{k}\right)$
such that $y_{4},\ldots,y_{n},z_{1},\ldots,z_{k}$ are any odd primes,
and $y_{1},y_{2},y_{3}$ are different odd primes. The choice of $\G_{\mbox{quad}}$
as in the above proof, however, demonstrates the idea of the proofs
to come.
\end{rem}

\section{Intertwined polynomials \label{sec: Related polynomials}}

Our goal in the present and the following sections is to formulate sufficient
conditions on the coefficients $F_{i}\left(y\right)$ in a variety
of the form \ref{eq: general form of variety},
so that they will satisfy the conditions of Theorem \ref{thm: Prime solutions variety},
implying that prime points Zariski-dense in this variety. In particular,
we wish to be able to control the $\gcd$'s of every $n$-sized subset
of $\left\{ G\left(z\right)-m,F_{1}\left(y\right),\ldots,F_{n}\left(y\right)\right\} $,
for a Zariski-dense subset of prime vectors $\left(y,z\right)$. As a result,
we are interested in the common prime factors of $\left\{ G\left(z\right)-m,F_{1}\left(y\right),\ldots,F_{n}\left(y\right)\right\} $.
For reasons that will be discussed in the next section, the case of
the prime $2$ should be handled separately. In this section, we formulate
conditions on a pair $F,\tilde{F}$ of polynomials, such that there
exists a Zariski-dense subset $\G'$ of prime vectors $y,z$ for which every
pair in $\left\{ G\left(z\right)-m,F\left(y\right),\tilde{F}\left(y\right)\right\} $
has no common prime factors other than $2$. In particular, if the
set $\left\{ F_{1},\ldots,F_{n}\right\} $ contains such a pair, then
for every $\left(y,z\right)\in\G'$ the $\gcd$ of every $n$-sized
subset of $\left\{ G\left(z\right)-m,F_{1}\left(y\right),\ldots,F_{n}\left(y\right)\right\} $
is a power of $2$.

Very briefly put, the property we define asserts that the polynomials in question have an "iterated linear structure", as follows. 
\begin{defn}[A pair of intertwined polynomials]
Let $\mathcal{Y}$ be a set of commutative variables. 
Two polynomials $F,\widetilde{F}$ in the variables $\mathcal{Y}$ with integer coefficients are called
\emph{intertwined of depth $d=1$}, if the set of variables $\mathcal{Y}$ has a decomposition to three mutually disjoint sets 
$$\mathcal{Y}=\mathcal{U}\cup \widetilde{\mathcal{U}}\cup \mathcal{W}=
\left\{ \left(u_{i}\right)_{i},\left(\tilde{u}_{i}\right)_{i},\left(w_{i}\right)_{i}\right\}$$
with $\mathcal{U}$ and $ \widetilde{\mathcal{U}}$ non-empty and of the same size (denoted $k_0$), and $F$, $\widetilde{F}$ of the form 

\begin{equation}
F\left(\mathcal{Y}\right)=\sum_{i=1}^{k_{0}}\a_{i}u_{i}+\b\left(\mathcal{W}\right),\,\tilde{F}\left(\mathcal{Y}\right)=\sum_{i=1}^{k_{0}}\a_{i}\tilde{u}_{i}+\b\left(\mathcal{W}\right)\label{eq: related of depth 1}
\end{equation}
where $\left\{ \a_{1},\ldots,\a_{k_{0}}\right\} $ are integers whose
$\gcd$  is a power of $2$ and $\b\left(w\right)$ is an arbitrary polynomial
with integer coefficients.

In particular, note that keeping the variables in $\mathcal{W}$ fixed, $F$ and $\widetilde{F}$ are linear forms in the two disjoint sets of variables $\mathcal{U}$ and $\widetilde{\mathcal{U}}$, inhomogeneous if $\mathcal{W}\neq \emptyset$ and $\beta\neq 0$.

Continuing inductively, two polynomials $F,\widetilde{F}$ in a set of commuting variables $\mathcal{Y}$ with integer coefficients are called
\emph{intertwined of depth $d\ge 2$}, if $\mathcal{Y}$ has a decomposition to four disjoint sets 

$$\mathcal{Y}=\mathcal{U}\cup \widetilde{\mathcal{U}}\cup \mathcal{V}\cup \mathcal{W}=
\left\{ \left(u_{i}\right)_{i},\left(\tilde{u}_{i}\right)_{i},\left(v_{i}\right)_{i},\left(w_{i}\right)_{i}\right\}$$
with $\mathcal{U}$ and $ \widetilde{\mathcal{U}}$ non-empty and of the same size (denoted $k_d$), and $F$, $\widetilde{F}$ of the form 
$$F\left(\mathcal{Y}\right)=\sum_{i=1}^{k_{d}}\a_{i}\left(\mathcal{V}\right)u_{i}+\b\left(\mathcal{V},\mathcal{W}\right),\,\tilde{F}\left(\mathcal{Y}\right)=\sum_{i=1}^{k_{d}}\a_{i}\left(\mathcal{V}\right)\tilde{u}_{i}+\b\left(\mathcal{V},\mathcal{W}\right)
$$
where for some pair $i_1, i_2$ the polynomials $\a_{i_1}(\mathcal{V}),\a_{i_2}(\mathcal{V})$ are intertwined of depth $d-1$ (namely the set of variables $\mathcal{V}$ itself has a decomposition into disjoint sets of variables with $\a_{i_1}(\mathcal{V}),\a_{i_2}(\mathcal{V})$ - playing the role of $F$ and $\widetilde{F}$ - satisfying the foregoing conditions).

Note that fixing the variables in $\mathcal{V}$ and $\mathcal{W}$ again $F$ and $\widetilde{F}$ are linear forms in the two disjoint sets of variables $\mathcal{U}$ and $\widetilde{\mathcal{U}}$, possibly inhomogeneous. 

We will also say that $F$, $\tilde{F}$ are intertwined through the
set of polynomials $\left\{ \a_{i}\right\} _{i=1}^{k_{d}}$.

 \end{defn}

%\begin{rem}
%Note that $F$, $\tilde{F}$ are the same polynomial in different
%variables: $\tilde{F}$ is obtained from $F$ by replacing every $u_{i}$
%with $\tilde{u}_{i}$. However, the definition of intertwined polynomials
%can be extended to the case where
%\begin{eqnarray*}
%F\left(y\right) & = & \a_{1}\left(v\right)u_{1}+\a_{2}\left(v\right)u_{2}+\sum_{i=3}^{k_{d}}\a_{i}\left(v\right)u_{i}+\b\left(v,w\right)\\
%\tilde{F}\left(y\right) & = & \a_{1}\left(v\right)\tilde{u}_{1}+\a_{2}\left(v\right)\tilde{u}_{2}+\sum_{i=3}^{k_{d}}\tilde{\a}_{i}\left(v\right)\tilde{u}_{i}+\tilde{\b}\left(v,w\right)
%\end{eqnarray*}
%with $\b,\a_{i}$ that may be different from $\tilde{\b},\tilde{\a}_{i}$
%for $i=3,\ldots,n$. For this extended definition, Theorem \ref{thm: related polynomials and prime points}
%below (as well as Theorem \ref{thm: General Thm}) would still hold.
%For simplicity, we only formulate the simpler case, since it is the
%only one that is being used in the present paper. \end{rem}
Before proceeding to give an example of intertwined polynomials, let us introduce the following 
\paragraph{Notation. $M_{i_{1},\ldots,i_{k}}^{j_{1},\ldots,j_{l}}\left(\protect\matA\right)$:}

For a matrix $\matA$, we let $M_{i_{1},\ldots,i_{k}}^{j_{1},\ldots,j_{l}}\left(\matA\right)$
denote the matrix obtained from $\matA$ by deleting the rows indexed
$i_{1},\ldots,i_{k}$ and the columns indexed $j_{1},\ldots,j_{l}$.

\begin{example}
\label{exa: det example for intertwined}If $\matA$ and $\matB$
are two $k\times k$ matrices of variables that are identical except for their
$j$-th row (resp. column), and in the $j$-th rows (resp. columns) the sets of variables that appear in $a$ and $b$ are disjoint. Then the polynomials $\det\left(\matA\right)$
and $\det\left(\matB\right)$ are intertwined of depth $k$. Indeed,
if $k=1$, then they are clearly intertwined of depth $1$; for $k>1$ we let $\mathcal{Y}$ denote the union of variables appearing in $a$ and $b$, and 
write
\[
F(\mathcal{Y})=\det\left(\matA\right)=\sum_{i=1}^{k}\left(-1\right)^{i+j}\det\left(M_{i}^{j}\left(\matA\right)\right)\cdot\matA_{i,j}
\]
and
\[
\widetilde{F}(\mathcal{Y})=\det\left(\matB\right)=\sum_{i=1}^{k}\left(-1\right)^{i+j}\det\left(M_{i}^{j}\left(\matB\right)\right)\cdot\matB_{i,j}=\sum_{i=1}^{k}\left(-1\right)^{i+j}\det\left(M_{i}^{j}\left(\matA\right)\right)\cdot\matB_{i,j}
\]
(when the matrices $a$ and $b$ differ by the $j$-th column). Then, setting $\mathcal{V}$ to be %=M^{j}\left(\matA\right)$ (namely $\mathcal{V}$ is
 the set of variables appearing in the 
matrix $M^{j}\left(\matA\right)$,
$\a_{i}\left(\mathcal{V}\right)=\left(-1\right)^{i+j}\det\left(M_{i}^{j}\left(\matA\right)\right)$,
$\b\equiv0$, $\mathcal{U}=\left(u_{i}\right)_{i=1}^{k}=\left(\matA_{i,j}\right)_{i=1}^{k}$
and $\widetilde{\mathcal{U}}=\left(\tilde{u}_{i}\right)_{i=1}^{k}=\left(\matB_{i,j}\right)_{i=1}^{k}$, we have 
:
\[
F(\mathcal{Y})=\det\left(\matA\right)=\sum_{i=1}^{k}\a_{i}\left(\mathcal{V}\right)u_{i},\:\widetilde{F}(\mathcal{Y})=\det\left(\matB\right)=\sum_{i=1}^{k}\a_{i}\left(\mathcal{V}\right)\tilde{u}_{i}
\]
where by the induction hypothesis, every pair among $\left\{ \a_{i}\left(\mathcal{V}\right)\right\} $
is intertwined of depth $k-1$, since they are determinants of two
$\left(k-1\right)\times\left(k-1\right)$ matrices that differ only
by one column (resp. row).
\end{example}
Since intertwined polynomials have integral
coefficients, they assume integer values on integral substitutions.
We are interested in the situation when the integral values obtained
by a intertwined couple have no common prime factors other than $2$.
\begin{defn}
Integers $\a_{1},\ldots,\a_{n}$ are called \emph{$2$-coprime} if
they have no common prime factors other than $2$, namely if $\gcd\left(\a_{1},\ldots,\a_{n}\right)$
is a non-negative power of $2$. In particular, if $\left\{ \a_{i}\right\} _{i=1}^{n}$
are coprime than they are $2$-coprime.
\end{defn}
 \global\long\def\linpol{f\left(y\right)}

\begin{lem}
\label{lem: create 2-coprime}Let $\a_{1},\ldots,\a_{n},\b,\ga$ be
non-zero integers, and let $s_{1},\ldots,s_{n}$ and $q_{1}\ldots,q_{n}$
be non-negative integers such that $\left\{ q_{i}\right\} _{i=1}^{n}$
are odd. Consider the following inhomogeneous linear form in the variables $y_{1},\ldots,y_{n}$:
\[
\linpol=\a_{1}y_{1}+\ldots+\a_{n}y_{n}+\b.
\]
If $\ga$ is 2-coprime to $\gcd\left(\a_{1},\ldots,\a_{n}\right)$,
then the set
\begin{equation}
\left\{ y=\left(y_{1},\ldots,y_{n}\right)\left|\begin{split} & y_{i}\mbox{ is an odd prime for all }i\\
 & y_{i}\equiv q_{i}\left(\mbox{mod }2^{s_{i}}\right)\mbox{ for all \ensuremath{i}}\\
 & \linpol\mbox{ is \ensuremath{2}-coprime to \ensuremath{\ga}}
\end{split}
\right.\right\} \subset\ZZ^{n}\label{eq:create 2-coprime}
\end{equation}
is Zariski-dense in $\F^{n}$. \end{lem}
\begin{rem}
Note that in the special case where $\a_{1},\ldots,\a_{n}$ are 2-coprime,
there are no restrictions on $\b$.
\end{rem}

\begin{rem}
The integer $\ga$ can be replaced by any finite number of non-zero
integers $\ga_{1},\ldots,\ga_{r}$, by taking $\ga=\lcm\left(\ga_{1},\ldots,\ga_{r}\right)$
(least common multiple); namely, if $\ga_{1},\ldots,\ga_{r}$ are
non-zero integers such that every $\ga_{j}$ is 2-coprime to $\gcd\left(\a_{1},\ldots,\a_{n}\right)$,
then the set
\[
\left\{ \left(y_{1},\ldots,y_{n}\right)\left|\begin{split} & y_{i}\mbox{ is an odd prime for all }i\\
 & y_{i}\equiv q_{i}\left(\mbox{mod }2^{s_{i}}\right)\mbox{ for all \ensuremath{i}}\\
 & \linpol\mbox{ is \ensuremath{2}-coprime to \ensuremath{\ga_{j}}\,\ for all \ensuremath{j}}
\end{split}
\right.\right\} \subset\ZZ^{n}
\]
is Zariski-dense in $\F^{n}$.
\end{rem}

The proof of Lemma \ref{lem: create 2-coprime} is postponed to the
Appendix. The concluding result of this section is the following.

\begin{thm}
\label{thm: related polynomials and prime points}Let $m\neq0$ be
an integer, and let $F,\tilde{F}$ be a pair of intertwined polynomials
of depth $d\geq1$ in the set of variables $\mathcal{Y}$. There exists a Zariski-dense
subset $\G'\subset\RR^{\left|\mathcal{Y}\right|}$ of odd prime points $y$
such that for every $y\in\G'$, any two integers in $\left\{ F\left(y\right),\tilde{F}\left(y\right),m\right\} $
are $2$-coprime.

Moreover, if $\left(q_{i},s_{i}\right)$ are non-negative integers
such that $q_{i}$ is odd, then the elements $y=\left(y_{i}\right)_{i}$
of $\G'$ can be chosen such that $y_{i}\equiv q_{i}\left(\mbox{mod }2^{s_{i}}\right)$
for every $i$.\end{thm}

\begin{proof}
By induction on the depth $d$. If $d=1$, then $F,\tilde{F}$ are
of the form \ref{eq: related of depth 1}:
\[
F\left(y\right)=\sum_{i=1}^{k_{0}}\a_{i}u_{i}+\b\left(w\right),\,\tilde{F}\left(y\right)=\sum_{i=1}^{k_{0}}\a_{i}\tilde{u}_{i}+\b\left(w\right),
\]
where $\gcd\left\{ \a_{i}\right\} _{i=1}^{k_{0}}$ is a power of $2$,
and in particular $2$-coprime to $m$. Hence, by Lemma \ref{lem: create 2-coprime},
for any integral $w$ there exists a subset
\[
{\cal A}_{w}=\left\{ u=\left(u_{1},\ldots,u_{k_{0}}\right)\left|
\begin{split} & \left\{ u_{i}\right\} \mbox{ are odd primes and $u_{i}\equiv q_{i}\left(\mbox{mod }2^{s_{i}}\right)$
for every $i$}\\
 & F\left(y\right)=F\left(u,w\right)\mbox{ is \ensuremath{2}-coprime to \ensuremath{m}}
\end{split}
\right.\right\} \subset\ZZ^{k_{0}}
\]
which is Zariski-dense in $\RR^{k_{0}}$. Also by this Lemma, for
every integral $w$ and $u\in{\cal A}_{w}$ there exists a subset
\[
{\cal A}_{u,w}=\left\{ \tilde{u}=\left(\tilde{u}_{1},\ldots,\tilde{u}_{k_{0}}\right)\left|\begin{split} & \left\{ \tilde{u}_{i}\right\} \mbox{ are odd primes and $\tilde{u}_{i}\equiv q_{i}\left(\mbox{mod }2^{s_{i}}\right)$ for every $i$}\\
 & \tilde{F}\left(y\right)=F\left(\tilde{u},w\right)\mbox{ is \ensuremath{2}-coprime to both}\\
 & \mbox{\ensuremath{ m }\,\,and \ensuremath{F\left(y\right)=F\left(u,w\right)}}
\end{split}
\right.\right\} \subset\ZZ^{k_{0}}
\]
which is also Zariski-dense in $\RR^{k_{0}}$. According to Lemma
\ref{lem: for each set there is set- dense} (the
case of Example \ref{exa: for each set there is a set - dense}),
the set
\[
\G'=\left\{ y=\left(u,\tilde{u},w\right)\left|\begin{split} & \left(w_{i}\right)\mbox{ are odd primes $w_{i}\equiv q_{i}\left(\mbox{mod }2^{s_{i}}\right)$
for every $i$}\\
 & u\in{\cal A}_{w}\mbox{ and \ensuremath{\tilde{u}\in{\cal A}_{u,w}}}
\end{split}
\right.\right\} \subset\ZZ^{\left|\mathcal{Y}\right|}
\]
 is Zariski-dense in $\RR^{\left|\mathcal{Y}\right|}$.

Let $d\geq 2$. Since the pair $\a_{1}\left(\mathcal{V}\right),\a_{2}\left(\mathcal{V}\right)$
is intertwined of depth $d-1$, there exists a Zariski-dense set $\G'_{\mbox{ind}}\subset\RR^{\left|\mathcal{V}\right|}$
of odd prime $v$ (satisfying any desired odd congruence conditions
modulo powers of $2$) such that for every $v\in\G'_{\mbox{ind}}$,
the integers $m,\a_{1}\left(v\right),\a_{2}\left(v\right)$ are pairwise
$2$-coprime. In particular, $\gcd\left\{ \a_{i}\left(v\right)\right\} _{i=1}^{k_{d}}$
is a power of $2$, hence $2$-coprime to any given integer, for every
$v\in\G'_{\mbox{ind}}$.

Repeat a similar argument with Lemma \ref{lem: create 2-coprime}
as with the case of $d=1$:
\begin{enumerate}
\item For every $v\in\G'_{\mbox{ind}}$ and integral $w$, there exists
a Zariski-dense subset ${\cal A}_{v,w}\subset\RR^{k_{d}}$ of odd
prime $u$ (in the desired arithmetic progressions modulo powers
of $2$) such that every $u\in{\cal A}_{v,w}$, $F\left(y\right)=F\left(u,v,w\right)$
is $2$-coprime to $m$.
\item For every $v\in\G'_{\mbox{ind}}$, integral $w$ and $u\in{\cal A}_{v,w}$,
there exists a Zariski-dense subset ${\cal A}_{u,v,w}\subset\RR^{k_{d}}$
of odd prime $\tilde{u}$ (in the desired arithmetic progressions
modulo powers of $2$) such that every $\tilde{u}\in{\cal A}_{u,v,w}$,
$\tilde{F}\left(y\right)=\tilde{F}\left(\tilde{u},v,w\right)$ is
$2$-coprime to both $m$ and $F\left(y\right)$.
\end{enumerate}
By Lemma \ref{lem: for each set there is set- dense}, the set
\[
\G'=\left\{ y=\left(u,\tilde{u},v,w\right)\left|\begin{split} & \left(v_{i}\right),\left(w_{i}\right)\mbox{ are odd primes congruent to}\\
 & \mbox{$ q_{i}\left(\mbox{mod }2^{s_{i}}\right)$
for every $i$,}\\
 & u\in{\cal A}_{v,w}\mbox{ and \ensuremath{\tilde{u}\in{\cal A}_{u,v,w}}}
\end{split}
\right.\right\} \subset\ZZ^{\left|\mathcal{Y}\right|}
\]
is Zariski-dense in $\RR^{\left|\mathcal{Y}\right|}$.

\end{proof}

\section{Congruence conditions  on the coefficients:\label{sec: Powers of 2}
Main Theorem }

In Section \ref{sec: Introduction and examples} we have presented
several examples of varieties $\vgen_{m}$ to which our method will be shown to apply,
and therefore prime points are Zariski-dense in $\vgen_{m}$ if and
only if $m$ is such that there exists an odd point in $\vgen_{m}$.
Observe that in all of these examples --- the varieties $\vdet_{m}$,
$\vquad_{m}$, $\vpf_{m}$, $\vrec_{m}$ and $\vper_{m}$ --- the
necessary and sufficient condition on $m$ is a congruence condition
modulo a power of $2$. This is not a coincidence: substituting an
odd point $x$ imposes such conditions on $m=\inpol\left(x\right)$,
as well as on the polynomials $F_{i}\left(y\right)$. For example,
consider the case of the variety $\vdet_{m}$ defined by $\det\left(x\right)=m$.
The determinant of an odd $x\in\mbox{Mat}_{n}$ is divisible by $2^{n-1}$,
since we can add the first row to the $n-1$
remaining rows and obtain $n-1$ even rows, without changing the determinant.
Note that this parity condition is the strongest that is shared by
all the odd $n\times n$ matrices, since modulo $2^{n}$, $\det\left(x\right)$
can be congruent to either $0$ or $2^{n-1}$ (depending on $x$).
\global\long\def\maxexp{\varepsilon}

Let us now return to the notation set in equation (\ref{eq: def of pol Delta}). Recall that $x=\left(\xi,y,z\right)$ and
\[
\inpol\left(x\right)=F_{1}\left(y\right)\xi_{1}+\cdots+F_{n}\left(y\right)\xi_{n}+G\left(z\right)
\]
 where $F_{i}\left(y\right)$ are non-zero polynomials 
and $G\left(z\right)$ are polynomials with integer coefficients.

\paragraph{Notation. }%$\protect\maxexp\left(y\right)$. }

For an integral $y=\left\{ y_{1},\ldots,y_{N}\right\} $, let $\maxexp\left(y\right)$
be the maximal positive integer such that $2^{\maxexp\left(y\right)-1}\mid F_{i}\left(y\right)$
for every $i=1,\ldots,n$.

\begin{mythm}\label{thm: General Thm}Let
$x$ and $\inpol\left(x\right)$ as above, and let $m\in\ZZ$. Assume
that $n\geq3$, that $G(z)$ is not identically equal to $ m$, and that there exist two polynomials
in $\left\{ F_{1},\ldots,F_{n}\right\} $ that are intertwined. Then,
the prime points are Zariski-dense in $\vgen_{m}=\left\{ x:\inpol\left(x\right)=m\right\} $
if and only if there exists an odd point $\odd x=\left(\odd{\xi},\odd y,\odd z\right)$
in $\ZZ^{n+N+k}$ such that  $m\equiv\inpol\left(\odd x\right)\left(\mbox{mod }2^{\maxexp\left(\odd y\right)}\right)$.

\end{mythm}

\begin{proof}
The condition on $m$ is necessary; if prime points are Zariski-dense
in $\vgen_{m}$, then there exists an odd point $\odd x$ in $\vgen_{m}$.
Otherwise, every prime point $x=\left(x_{i}\right)_{i=1}^{n+N+k}\in\vgen_{m}$
has $2$ as one of its entries. In particular, the prime points in
$\vgen_{m}$ are contained in $Z\left(\prod_{i=1}^{n+N+k}\left(x_{i}-2\right)\right)$,
and are therefore not Zariski-dense in $\vgen_{m}$, unless
\[
\vgen_{m}=Z\left(\Delta-m\right)\subseteq Z\left(\prod_{i=1}^{N+n+k}\left(x_{i}-2\right)\right).
\]
However, the latter is impossible, since $\Delta-m$ is irreducible
and therefore the above implies that
\[
\left(\Delta-m\right)\mid\prod_{i=1}^{n+N+k}\left(x_{i}-2\right),
\]
i.e. $\Delta-m=x_{i}-2$ for some $i=1,\ldots,n+N+k$, a
contradiction.  We conclude that there exists an odd point $\odd x$
in $\vgen_{m}$, and in particular $m\equiv\inpol\left(\odd x\right)\left(\mbox{mod }2^{\maxexp\left(\odd y\right)}\right)$.

For sufficiency, we apply Theorem \ref{thm: Prime solutions variety}.
Assume that there exists an odd $\odd x$ in $\ZZ^{n+N+k}$
such that $m\equiv\inpol\left(\odd x\right)\left(\mbox{mod }2^{\maxexp\left(\odd y\right)}\right)$,
and let $\maxexp=\maxexp\left(\odd y\right)$. We show that  there
exists a Zariski-dense set of odd prime $\left(y,z\right)$ such that
\begin{enumerate}
\item $\sum_{i=1}^{n}F_{i}\left(y\right)+G\left(z\right)\equiv m\left(\mbox{mod }2^{\maxexp}\right)$,

\item the $\gcd$ of every $n$-sized subset of $\left\{ G\left(z\right)-m,F_{1}\left(y\right),\ldots,F_{n}\left(y\right)\right\} $
equals $2^{\maxexp-1}$.
\end{enumerate}

Hence the conditions of Theorem \ref{thm: Prime solutions variety}
are satisfied, and prime points are Zariski-dense in $\vgen_{m}$.
We begin by showing that the odd point $\odd x=\left(\odd{\xi},\odd y,\odd z\right)$
(which is not necessarily in the variety $\vgen_{m}$!) satisfies
the first condition, and the following weakened form of the second
condition:
\begin{itemize}
\item [$\odd{2}.$]the $\gcd$ of every $n$-sized subset of $\left\{ G\left(\odd z\right)-m,F_{1}\left(\odd y\right),\ldots,F_{n}\left(\odd y\right)\right\} $
is congruent to $2^{\maxexp-1}\left(\mbox{mod }2^{\maxexp}\right)$.
\end{itemize}
By assumption,
\[
m\equiv\sum_{i=1}^{n}F_{i}\left(\odd y\right)\odd{\xi_{i}}+G\left(\odd z\right)\left(\mbox{mod }2^{\maxexp}\right),
\]
where $2^{\maxexp-1}$ divides every $F_{i}\left(\odd y\right)$,
and therefore divides $m-G\left(\odd z\right)$ as well. Hence,
\[
\frac{m-G\left(\odd z\right)}{2^{\maxexp-1}}\equiv\sum_{i=1}^{n}\frac{F_{i}\left(\odd y\right)}{2^{\maxexp-1}}\cdot\odd{\xi_{i}}\left(\mbox{mod }2\right).
\]
Since $\left\{ \odd{\xi_{i}}\right\} $ are odd, then in particular
\[
\frac{m-G\left(\odd z\right)}{2^{\maxexp-1}}\equiv\sum_{i=1}^{n}\frac{F_{i}\left(\odd y\right)}{2^{\maxexp-1}}\left(\mbox{mod }2\right).
\]
Multiply by $2^{\maxexp-1}$ to obtain
\begin{equation}
m-G\left(\odd z\right)\equiv\sum_{i=1}^{n}F_{i}\left(\odd y\right)\left(\mbox{mod }2^{\maxexp}\right),\label{eq: m equivalent mod 2^a to sum of F_i}
\end{equation}
which means that $\odd x=\left(\odd{\xi},\odd y,\odd z\right)$ indeed
satisfies the first condition.

For condition $\odd 2$, observe that
equation \ref{eq: m equivalent mod 2^a to sum of F_i} implies
\begin{equation}
G\left(\odd z\right)-m+\sum_{i=1}^{n}F_{i}\left(\odd y\right)\equiv0\left(\mbox{mod }2^{\maxexp}\right).\label{eq: sum is even}
\end{equation}
By the choice of $\maxexp$, there exists some $i_{1}\in\left\{ 1,\ldots,n\right\} $
such that $F_{i_{1}}\left(\odd y\right)\equiv2^{\maxexp-1}\left(\mbox{mod }2^{\maxexp}\right)$.
Hence equation \ref{eq: sum is even} implies that there exists another
summand among
\[
\left\{ G\left(\odd z\right)-m,F_{1}\left(\odd y\right),\ldots,F_{n}\left(\odd y\right)\right\}
\]
which is also congruent to $2^{\maxexp-1}\left(\mbox{mod }2^{\maxexp}\right)$.
Thus, the $\gcd$ of every $n$-sized subset of $\left\{ G\left(\odd z\right)-m,F_{1}\left(\odd y\right),\ldots,F_{n}\left(\odd y\right)\right\} $
is congruent to $2^{\maxexp-1}\left(\mbox{mod }2^{\maxexp}\right)$.

Assume that $\left\{ F,\tilde{F}\right\} \subset\left\{ F_{i}\right\} _{i=1}^{n}$
are intertwined, and let $\G'$ be a Zariski-dense set of odd prime
$y$ such that $y\equiv\odd y\left(\mbox{mod }2^{\maxexp}\right)$
and such that every two integers in $\left\{ m-G\left(\odd z\right),F\left(y\right),\tilde{F}\left(y\right)\right\} $
are $2$-coprime (the existence of $\G'$ was establish in Theorem
\ref{thm: related polynomials and prime points}). 

Finally, note that $y\equiv\odd y\left(\mbox{mod }2^{\maxexp}\right)$ implies that 
$P(y)\equiv P(\odd y)\left(\mbox{mod }2^{\maxexp}\right)$ for any polynomial $P$.  There is a Zariski dense subset of odd prime $z$ satisfying the congruence condition $z\equiv\odd z\left(\mbox{mod }2^{\maxexp}\right)$, and they also satisfy $m-P(z)\equiv m-P(\odd z)\left(\mbox{mod }2^{\maxexp}\right)$ (for any $m$). 
Therefore the properties established above for $\odd y, \odd z$ imply that 
every $y\in\G'$ and odd prime $z\equiv\odd z\left(\mbox{mod }2^{\maxexp}\right)$  satisfy conditions 1 and 2 above.
\end{proof}
In all the examples we consider (Theorems \ref{thm: determinant variety},
\ref{thm: quadratic form variety}, \ref{thm: Pfaffian variety},
\ref{thm: Rectangular matrices variety} and \ref{thm: permanent variety}),
the polynomials $\inpol\left(x\right)$ and $\left\{ F_{j}\left(y\right)\right\} _{j=1}^{n}$
have additional symmetry, yielding a situation where the congruence conditions
on $m$ required in Theorem \ref{thm: General Thm} for \emph{some}
odd point $\odd x=\left(\odd{\xi},\odd y,\odd z\right)$ in $\ZZ^{n+N+k}$,
are actually satisfied by \emph{all} the odd points in $\ZZ^{n+N+k}$.
Thus, in the proofs of the above mentioned theorems, we shall use
the following special case of Theorem \ref{thm: General Thm}.

\begin{mythm}\label{thm: Weak General Thm}

In the setting of Theorem \ref{thm: General Thm}, let $\maxexp$
be the maximal positive integer such that $2^{\maxexp-1}\mid F_{i}\left(y\right)$
for every $i=1,\ldots,n$ and every odd $y=\left\{ y_{1},\ldots,y_{N}\right\} $
 and assume that for any odd $\odd x,\odd{\odd{x}}$ it holds that
\[
\inpol\left(\odd x\right)\equiv\inpol\left(
\odd{\odd x}\right)\left(\mbox{mod }2^{\maxexp}\right).
\]
 Then prime points are Zariski-dense in $\vgen_{m}$ if and only
if  $m\equiv\inpol\left(x\right)\left(\mbox{mod }2^{\maxexp}\right)$
for some (and actually, every) odd $x$.

\end{mythm}

\begin{proof}
As in the proof of Theorem \ref{thm: General Thm}: if prime points
are Zariski-dense in $\vgen_{m}$, then there exists an odd point
$\odd x$ in $\vgen_{m}$, and in particular $m\equiv\inpol\left(\odd x\right)\left(\mbox{mod }2^{\maxexp\left(\odd y\right)}\right)$.
Since $\maxexp\leq\maxexp\left(\odd y\right)$, $m\equiv\inpol\left(\odd x\right)\left(\mbox{mod }2^{\maxexp}\right)$.

Conversely, assume $m\equiv\inpol\left(x\right)\left(\mbox{mod }2^{\maxexp}\right)$
for every odd $x$. Let $\odd y\in\ZZ^{N}$ be odd such
that $2^{\maxexp}\nmid F_{i}\left(\odd y\right)$ for some $i\in\left\{ 1,\ldots,n\right\} $;
then $\maxexp\left(\odd y\right)=\maxexp$. By assumption, for any
odd $\odd{\xi}$ and $\odd z$, $m\equiv\inpol\left(\odd x\right)\left(\mbox{mod }2^{\maxexp}\right)$
where $\odd x=\left(\odd{\xi},\odd y,\odd z\right)$. Then  the conditions
of Theorem \ref{thm: General Thm} are met, and prime points are Zariski-dense
in $\vgen_{m}$. \end{proof}
\begin{example}
We note that Theorem \ref{thm: determinant variety}
for Zariski-density of prime points in $\vdet_{m}=\left\{ x\in\mbox{Mat}_{n}:\det\left(x\right)=m\right\} $
is a consequence of Theorem \ref{thm: Weak General Thm}. For $x\in\mbox{Mat}_{n}\left(\ZZ\right)$
denote the first row of $x$ by $\left(\xi_{1},\ldots,\xi_{n}\right)$,
and the matrix obtained by removing the first row of $x$ by $y\in\mbox{Mat}_{\left(n-1\right)\times n}\left(\ZZ\right)$.
An expansion of $\det\left(x\right)$ along the first row yields
\[
D_{1}\left(y\right)\xi_{1}+\cdots+D_{n}\left(y\right)\xi_{n}=m,
\]
where $\left\{ D_{i}\left(y\right)\right\} _{i=1}^{n}$ are polynomials
in the entries of $y$, and more specifically, determinants of $\left(n-1\right)\times\left(n-1\right)$
submatrices of $y$ (with alternating signs). By Example \ref{exa: det example for intertwined},
any pair of coefficients $D_{i_{1}}\left(y\right),D_{i_{2}}\left(y\right)$
is intertwined. As explained in the beginning of this section, for
every odd $\odd x$ and $i\in\left\{ 1,\ldots,n\right\} $ it holds
that $\det\left(\odd x\right)\equiv0\left(\mbox{mod }2^{n-1}\right)$
and $D_{i}\left(\odd y\right)\equiv0\left(\mbox{mod }2^{n-2}\right)$,
and it can be shown that these powers are the maximal that hold for
every odd point; then $\maxexp=n-1$ and by Theorem \ref{thm: Weak General Thm},
prime points are Zariski-dense in $\vdet_{m}$
if and only if $m\equiv0\left(\mbox{mod }2^{n-1}\right)$.
\end{example}

\begin{example}
Theorem \ref{thm: quadratic form variety} for Zariski density of
prime points in $\vquad_{m}$, defined as a level set of the quadratic
form $Q{}_{n,k}\left(\xi,y,z\right)=\sum_{i=1}^{n}\xi_{i}y_{i}+\sum_{i=1}^{k}z_{i}^{2}$,
is also a consequence of Theorem \ref{thm: Weak General Thm}. With
$F_{i}\left(y\right)=y_{i}$ for all $i=i,\ldots,n$ and $G\left(z\right)=\sum_{i=1}^{k}z_{i}^{2}$
(as in the proof of Theorem \ref{thm: quadratic form variety}, Section
\ref{sec: Quadratic forms}), it is clearly the case that $\maxexp=1$
 and $\inpol\left(x\right)=Q{}_{n,k}\left(\xi,y,z\right)\equiv m\equiv0\left(\mbox{mod }2^{1}\right)$
if and only if $n+k\equiv0\left(\mbox{mod }2\right)$.
\end{example}
We now proceed to prove Theorems \ref{thm: Pfaffian variety}, \ref{thm: Rectangular matrices variety}
and \ref{thm: permanent variety} by verifying that the varieties
in question satisfy the conditions of Theorem \ref{thm: Weak General Thm}.

\section{Variety of anti-symmetric matrices of fixed Pfaffian\label{sec: Anti -symmetric odd Pfaffian}}

Denote by $\asym_{2n}\left(\RR\right)$ the space of anti-symmetric
matrices of order $2n\times2n$ over $\RR$, and recall
$M_{i_{1},\ldots,i_{k}}^{j_{1},\ldots,j_{l}}\left(\matA\right)$ denotes
the matrix obtained from a matrix $\matA$ by deleting the rows indexed
$i_{1},\ldots,i_{k}$ and the columns indexed $j_{1},\ldots,j_{l}$.
The \textit{Pfaffian} of a matrix in $\asym_{2n}\left(\F\right)$
is a polynomial of degree $n$ in the matrix entries that can be defined
recursively. By convention, the Pfaffian of the $0\times0$ matrix
is defined to be $1$. For $n\geq1$ let $x\in\asym_{2n}\left(\mathbb{R}\right)$:
\[
x=\left[\begin{array}{cccc}
0 & x_{12} & \cdots & x_{1\,2n}\\
-x_{12}\\
\vdots &  & \ddots & \vdots\\
-x_{1\,2n} &  & \cdots & 0
\end{array}\right],
\]
and observe that $M_{i,j}^{i,j}\left(x\right)$ is a $2\left(n-1\right)\times2\left(n-1\right)$
anti-symmetric matrix; then
\begin{equation}
\pf\left(x\right)=\sum_{j=2}^{2n}\left(-1\right)^{j}x_{1j}\pf\left(M_{1,j}^{1,j}\left(x\right)\right).\label{eq: Pfaffian formula}
\end{equation}
For example,
\[
\pf\left(\left[\begin{array}{cc}
0 & a\\
-a & 0
\end{array}\right]\right)=a,
\]
\[
\pf\left(\left[\begin{array}{cccc}
0 & a & b & c\\
-a & 0 & d & e\\
-b & -d & 0 & f\\
-c & -e & -f & 0
\end{array}\right]\right)=af-be+cd.
\]

The group $\mbox{GL}_{2n}\left(\RR\right)$
acts on $\asym_{2n}\left(\mathbb{R}\right)$ by matrix congruence:
$g\cdot x=gxg^{T}$, and the non-singular matrices in $\asym_{2n}\left(\mathbb{R}\right)$
are a single orbit. To see that, observe that every non-singular $2n\times2n$
anti-symmetric matrix is congruent to $\Omega_{n}=\left[\begin{smallmatrix}0_{n} & I_{n}\\
-I_{n} & 0_{n}
\end{smallmatrix}\right]$. A theorem by Cayley  (\cite{Cayley_1847}, \cite{Muir_1911}) states
that for every $x\in\asym_{2n}\left(\mathbb{R}\right)$,
\begin{equation}
\det\left(x\right)=\left(\pf\left(x\right)\right)^{2}.\label{eq: det=00003DPf^2}
\end{equation}
It follows that for every $x\in\asym_{2n}\left(\mathbb{R}\right)$
and $g\in\mbox{GL}_{2n}\left(\mathbb{R}\right)$:
\begin{equation}
\pf\left(g^{\mbox{t}}xg\right)=\det\left(g\right)\pf\left(x\right).\label{eq: Pfaffian identity}
\end{equation}
In particular, the Pfaffian is an invariant  for the action of $\mbox{SL}_{2n}\left(\mathbb{\RR}\right)$
on $\asym_{2n}\left(\RR\right)$, and the orbits of $\mbox{SL}_{2n}\left(\mathbb{\RR}\right)$
on the non-singular matrices in $\asym_{2n}\left(\RR\right)$ are
the level sets of the Pfaffian:
\[
\vpf_{m}=\left\{ x\in\asym_{2n}\left(\RR\right)\mid\pf\left(x\right)=m\right\}
\]
with $m\neq0$.  Call a matrix in $\vpf_{m}$ prime if all its non-diagonal entries are primes ($\neq 1$) in $\mathbb{Z}$.  The goal of this section is to prove the following:

\def\thethmrpt{\ref{thm: Pfaffian variety}}
\begin{thmrpt}

For $n\geq2$, prime matrices are Zariski-dense in $\vpf_{m}\subset\asym_{2n}$
if and only if $m$ is an odd integer.

\end{thmrpt}

If prime matrices are Zariski-dense in $\vpf_{m}$, then there exists
an odd point $x\in\vpf_{m}$, which is a matrix $x$ whose non-diagonal
entries are odd integers. The necessity of the condition on $m$ is
then a consequence of the following.
\begin{lem}
\label{lem: Pf of odd matrix is odd}The Pfaffian of an odd anti-symmetric
matrix, which is a matrix whose non-diagonal entries are odd integers,
is an odd integer.\end{lem}
\begin{proof}
Let
\[
x=\left[\begin{array}{cccc}
0 & x_{12} & \cdots & x_{1\,2n}\\
-x_{12} & 0\\
\vdots &  & \ddots & \vdots\\
-x_{1\,2n} &  & \cdots & 0
\end{array}\right]
\]
 where $n\geq1$ and $x_{ij}$ is odd for all $i\neq j$. We prove
by induction on $n$.

If $n=1$ then $x$ is of the form
\[
x=\left[\begin{array}{cc}
0 & q\\
-q & 0
\end{array}\right]
\]
where $q$ is an odd integer, and $\pf\left(x\right)=q$ .

Let $n>1$, and assume the claim holds for $n-1$. For $i\neq j$,
the matrix $M_{i,j}^{i,j}\left(x\right)$ lies in $\asym_{2\left(n-1\right)}\left(\F\right)$,
and therefore $\pf\left(M_{i,j}^{i,j}\left(x\right)\right)$ is odd.
It follows that
\[
\pf\left(x\right)=\sum_{j=2}^{2n}\left(-1\right)^{j}x_{1j}\pf\left(M_{1,j}^{1,j}\left(x\right)\right)
\]
is odd, since each summand is odd  and there is an odd number $2n-1$
of summands.
\end{proof}
We now turn to prove Theorem \ref{thm: Pfaffian variety}.
\begin{proof}[Proof of Theorem \ref{thm: Pfaffian variety}]
If $x$ is a matrix in $\asym_{2n}\left(\F\right)$, write $\left(0,x_{1,2},\ldots,x_{1,2n}\right)=\left(\xi_{2},\ldots,\xi_{2n}\right)$
for its first row (and column), and $y\in\asym_{2n-1}\left(\F\right)$
for the matrix obtained from $x$ by deleting its first row and column:
\[
x=\left[\begin{array}{cc}
0 & \begin{array}{ccc}
\xi_{2} & \cdots & \xi_{2n}\end{array}\\
\begin{array}{c}
-\xi_{2}\\
\vdots\\
-\xi_{2n}
\end{array} & \begin{array}{|ccc}
\hline \,\,\,\, &  & \,\,\,\,\\
 & y\\
\\
\end{array}
\end{array}\right].
\]
Denote by $P_{j}\left(y\right)$ the Pfaffian of the $\left(2n-2\right)\times\left(2n-2\right)$
anti-symmetric matrix obtained from $y$ by deleting its $j$-th row
and column, namely, $P_{j}\left(y\right)=\pf\left(M_{j}^{j}\left(y\right)\right)=\pf\left(M_{1,j+1}^{1,j+1}\left(x\right)\right)$.
Then,
\[
\pf\left(x\right)=\sum_{j=2}^{2n}\left(-1\right)^{j}\xi_{j}P_{j-1}\left(y\right)=P_{1}\left(y\right)\xi_{2}-P_{2}\left(y\right)\xi_{3}+...+\left(-1\right)^{2n}P_{2n-1}\left(y\right)\xi_{2n}
\]
and in particular $\pf\left(x\right)=\inpol\left(x\right)$ is of
the form \ref{eq: def of pol Delta}.

Formula \ref{eq: Pfaffian formula} for the Pfaffian and the fact
that it is defined recursively imply that the Pfaffians of two anti-symmetric
matrices that differ only in their first row and column are intertwined.
Hence $P_{1}\left(y\right),P_{2}\left(y\right)$ are intertwined.

By Lemma \ref{lem: Pf of odd matrix is odd},  when $x$ is odd,
then so are $\pf\left(x\right)$ and the $P_{i}\left(y\right)$'s;
the conditions of Theorem \ref{thm: Weak General Thm} are therefore
satisfied with $\maxexp=1$. In particular,
prime points are Zariski-dense in $\vpf_{m}$ if and only if $m\equiv\inpol\left(x\right)\left(\mbox{mod }2\right)$
for every odd $x$, namely if and only if $m\equiv1\left(\mbox{mod 2}\right)$. \end{proof}

\begin{rem}
\label{rem: hafnian}An analog for the Pfaffian of anti-symmetric
matrices of even order is defined for symmetric matrices of even order
whose main diagonal is identically zero; it is called the \emph{hafnian},
and is defined as follows. For
\[
x=\left[\begin{array}{cccc}
0 & x_{1,2} & \cdots & x_{1,2n}\\
x_{1,2}\\
\vdots &  & \ddots & \vdots\\
x_{1,2n} &  & \cdots & 0
\end{array}\right],
\]
the hafnian of $x$ is
\[
\haf\left(x\right)=\sum_{j=2}^{2n}x_{1j}\cdot\haf\left(M_{1,j}^{1,j}\left(x\right)\right),
\]
where
\[
\haf\left(\left[\begin{array}{cc}
0 & q\\
q & 0
\end{array}\right]\right)=q.
\]
In other words, the hafnian polynomial is obtained from the Pfaffian
by switching all the negative signs to positive ones. In this sense,
it is analogous to the \emph{permanent} of a square matrix (see Section
\ref{sec: non-homogenous varieties}). An identical proof to the one
of Theorem \ref{thm: Pfaffian variety} yields that prime matrices
are Zariski-dense in the variety of fixed hafnian $m\neq0$ if and
only if $m$ is odd. This variety is not invariant under a group action,
namely it is non-homogeneous.
\end{rem}

\section{Variety of rectangular matrices\label{sec: rectangular matrices}}

\subsection{Motivation from pre-homogeneous vector spaces}

In all the examples we considered so far (except for the variety $\vper_{m}$
and the hafnian variety, which are non-homogeneous), the varieties
$\vgen_{m}\left(\RR\right)$ for $0\neq m\in\RR$ are $\mbox{SL}_{n}\left(\F\right)$
orbits, and foliate an open orbit of $\mbox{GL}_{n}\left(\F\right)$:
\begin{enumerate}
\item The varieties $\vdet_{m}$ foliate the open $\mbox{GL}_{n}\left(\F\right)$-orbit
$\left\{ x\in\mbox{Mat}_{n}\left(\RR\right)\mid\det\left(x\right)\neq0\right\} $;
\item the varieties $\vpf_{m}$ foliate the open $\mbox{GL}_{2n}\left(\F\right)$-orbit
$\left\{ x\in\asym_{2n}\left(\RR\right)\mid\pf\left(x\right)\neq0\right\} $;
\item the varieties $\vquad_{m}$ foliate the open $G_{n,k}\times\mbox{GL}_{1}\left(\F\right)$-orbit
$\left\{ x\in\RR^{2n+k}\mid Q_{n,k}\left(x\right)\neq0\right\} $.
\end{enumerate}

The pairs $\left(\mbox{GL}_{n}\left(\F\right),\mbox{Mat}_{n}\left(\F\right)\right)$,
$\left(\mbox{GL}_{2n}\left(\F\right),\asym_{2n}\left(\F\right)\right)$,
and $\left(G_{n,k}\times\mbox{GL}_{1}\left(\RR\right),\RR^{2n+k}\right)$
are therefore examples of \emph{pre-homogeneous vector spaces:}

\begin{defn}[\cite{PV_book}; see also \cite{SSM_90}, \cite{Servedio_73}. ]
Let $G$ be a connected linear algebraic group over an algebraically
closed field $\KK$ and let $V$ be a finite-dimensional vector space
over $\KK$ which affords a rational representation of $G$. The pair
$\left(G,V\right)$ is called a \emph{pre-homogeneous vector space}
(or P.V., for short) if $G$ has a Zariski-open (and therefore Zariski-dense)
orbit in $V$.
\end{defn}
Sato and Kimura have classified the irreducible pre-homogeneous
vector spaces in \cite{PHVS_77}. According to this classification,
there are only five infinite families of regular irreducible P.V.s,
and the remaining P.V.s are exceptional cases. Two of these families
are $\left(\mbox{GL}_{2n}\left(\F\right),\asym_{2n}\left(\F\right)\right)$
and $\left(\mbox{GL}_{n}\left(\F\right),\mbox{Mat}_{n}\left(\F\right)\right)$;
the pair $\left(G_{n,k}\times\mbox{GL}_{1}\left(\RR\right),\RR^{2n+k}\right)$
 is a sub-family of a third family. In this section we consider a
fourth family, and provide a necessary and sufficient condition for
Zariski density of prime points in the level sets defined by the associated
invariant polynomial.

\subsection{Exposition of the example}

For $\l\geq n\geq1$, we consider the action of $\mbox{SP}_{\l}\left(\F\right)\times\mbox{GL}_{2n}\left(\F\right)$
on the space $\mbox{Mat}_{\,2\l\times2n}\left(\F\right)$ given by
\[
\left(g,h\right)\cdot x=gxh^{\mbox{t}}
\]
where $x\in\mbox{Mat}_{2\l\times2n}\left(\F\right)$, $g\in\mbox{SP}_{\l}\left(\F\right)$
and $h\in\mbox{GL}_{2n}\left(\F\right)$. Define the polynomial
\[
\pitz\left(x\right)=\pf\left(x^{\mbox{t}}\Omega_{\l}x\right),
\]
where $\Omega_{\l}$ is as defined in \ref{eq: Omega def}. The set
of matrices $x$ for which $\pitz\left(x\right)\neq0$ (equivalently,
$\det\left(x^{\mbox{t}}x\right)\neq0$) is an open orbit of $\mbox{SP}_{\l}\left(\F\right)\times\mbox{GL}_{2n}\left(\F\right)$,
and it is foliated by orbits of $\mbox{SP}_{\l}\left(\F\right)\times\mbox{SL}_{2n}\left(\F\right)$,
given by the level sets of $\pitz$:
\[
\vrec_{m}=\left\{ x\in\mbox{Mat}_{\,2\l\times2n}\left(\F\right)\mid\pitz\left(x\right)=m\right\}
\]
with $m\neq0$. Indeed, $\pitz$ is invariant under the action of
$\mbox{SP}_{\l}\left(\F\right)\times\mbox{SL}_{2n}\left(\F\right)$,
since, by \ref{eq: Pfaffian identity}:
\[
\pitz\left(\left(g,h\right)\cdot x\right)=\pitz\left(x\right)
\]
 (see \cite{PHVS_77}, \cite{Pitzian}).  A necessary and sufficient
condition for Zariski density of prime matrices in $\vrec_{m}$ is
as follows:

\def\thethmrpt{\ref{thm: Rectangular matrices variety}}
\begin{thmrpt}

For $n\geq1$, $\l\geq2$, $\l\geq n$, and $0\neq m\in\ZZ$, prime
matrices are Zariski-dense in $\vrec_{m}\subset\mbox{Mat}_{\,2\l\times2n}\left(\F\right)$
if and only if $m\equiv0\left(\mbox{mod }2^{2n-1}\right)$.

\end{thmrpt}

Theorem \ref{thm: Rectangular matrices variety} is also a consequence
of Theorem \ref{thm: Weak General Thm};
in particular, the proof relies on the fact that the polynomial $\pitz\left(x\right)$
is of the form \ref{eq: def of pol Delta}, with coefficients that
are intertwined, as detailed below.

\subsection{Proof of Theorem \ref{thm: Rectangular matrices variety} }

Throughout this section, for $i\in\left\{ 1,\ldots,2\l\right\} $
we denote $\hat{i}:=\left(i+\l\right)\left(\mod\,2\l\right)$. For
$x\in\mbox{Mat}_{\,2\l\times2n}\left(\F\right)$,
let $\left(\xi_{1},\ldots,\xi_{2\l}\right)^{\mbox{t}}$ denote the
last column of $x$ and let $y\in\mbox{Mat}_{2\l\times\left(2n-1\right)}\left(\F\right)$
denote the matrix obtained from $x$ by deleting its last column.
A key ingredient in the proof of Theorem
\ref{thm: Rectangular matrices variety} is that the form $\pf\left(x^{\mbox{t}}\Omega x\right)$
can be expanded along every column of $x$, and in particular along
the last column:
\begin{equation}
\pitz\left(x\right)=\sum_{i=1}^{2\l}B_{i}\left(y\right)\xi_{i}=B_{1}\left(y\right)\xi_{1}+\ldots+B_{2\l}\left(y\right)\xi_{2\l},\label{eq: P(x) along last column}
\end{equation}
where the coefficients $B_{i}\left(y\right)$ are polynomials in the
entries of $y$ given by:
\begin{eqnarray}
B_{i}\left(y\right) & = & \sum_{k=1}^{2n-1}\left(-1\right)^{k}\cdot x_{\hat{i},k}\cdot\pitz\left(M_{i,\hat{i}}^{k,2n}\left(x\right)\right)\nonumber \\
 & = & \sum_{k=1}^{2n-1}\left(-1\right)^{k}\cdot x_{\hat{i},k}\cdot\pitz\left(M_{i,\hat{i}}^{k}\left(y\right)\right)\label{eq: A_i(y) expressed through Pitzians}
\end{eqnarray}
 (recall the notation $M_{i_{1},\ldots,i_{k}}^{j_{1},\ldots,j_{l}}$
introduced at the end of Section \ref{sec: Introduction and examples}).
This, along with some further facts that we shall utilize on the structure
of $B_{i}\left(y\right)$, is proved in the short note \cite{Pitzian}.

For a matrix $x'$ with $2\l$ rows (such as $x$ and $y$) and an
integer $1\leq k\leq\l$, we define the following matrix:
\[
Z_{t_{1},\ldots,t_{k}}\left(x'\right):=\left(\begin{smallmatrix}-- & R_{t_{1}}\left(x'\right) & --\\
-- & R_{\widehat{t_{1}}}\left(x'\right) & --\\
 & \vdots\\
-- & R_{t_{k}}\left(x'\right) & --\\
-- & R_{\widehat{t_{k}}}\left(x'\right) & --
\end{smallmatrix}\right),
\]
where $\left\{ t_{1},\ldots,t_{k}\right\} \subset\left\{ 1,\ldots,\l\right\} $
are such that $t_{1}<\ldots<t_{k}$. Then $Z_{t_{1},\ldots,t_{k}}\left(x'\right)$
has $2k$ rows, and the same number of columns as $x'$.

In order to apply Theorem \ref{thm: Weak General Thm}, we begin by
establishing parity conditions on the coefficients $B_{i}\left(y\right)$.
\begin{lem}
\label{lem: parity conditinos on rectangular coefficients}Let $x\in\mbox{Mat}_{\,2\l\times2n}\left(\ZZ\right)$
be odd, and $y=M^{2n}\left(x\right)$. Fix $i\in\left\{ 1,\ldots,2\l\right\} $
and let $B_{i}\left(y\right)$ be as in \ref{eq: P(x) along last column}.
Then $B_{i}\left(y\right)\equiv0\left(\mbox{mod }2^{2n-2}\right)$,
and it can be either $0\left(\mbox{mod }2^{2n-1}\right)$ or $2^{2n-2}\left(\mbox{mod }2^{2n-1}\right)$;
namely, there exist odd $y^{0},y^{1}\in\mbox{Mat}_{\,2\l\times\left(2n-1\right)}\left(\ZZ\right)$
for which $B_{i}\left(y^{0}\right)\equiv0\left(\mbox{mod }2^{2n-1}\right)$
and $B_{i}\left(y^{1}\right)\equiv2^{2n-2}\left(\mbox{mod }2^{2n-1}\right)$.
\end{lem}

\begin{proof}
The polynomial $B_{i}\left(y\right)$ can
be presented as the sum of determinants of $\left(2n-1\right)\times\left(2n-1\right)$
sub-matrices of $y$ as follows (\cite{Pitzian}):
\begin{equation}
B_{i}\left(y\right)=\sum_{\substack{t_{1}<\ldots<t_{n-1}\\
\in\left\{ 1,\ldots,m\right\} \setminus\left\{ i\right\}
}
}\det\left(\begin{smallmatrix}R_{\widehat{i}}\left(y\right)\\
Z_{t_{1},\ldots,t_{n-1}}\left(y\right)
\end{smallmatrix}\right).\label{eq: A_i(y) is sum of determinants}
\end{equation}
Each determinant in this sum is of an odd matrix, and is therefore
divisible by $2^{2n-2}$; thus, $B_{i}\left(y\right)\equiv0\left(\mbox{mod }2^{2n-2}\right)$.

Fix $t_{1}<\ldots<t_{n-1}\in\left\{ 1,\ldots,\l\right\} \setminus\left\{ i\right\} $
and define $y^{1}\in\mbox{Mat}_{\,2\l\times\left(2n-1\right)}\left(\ZZ\right)$
with the following two properties. Firstly,
\[
R_{j}\left(y^{1}\right)\equiv R_{\hat{j}}\left(y^{1}\right)\left(\mbox{mod }2^{2n-2}\right)\mbox{ for every }j\notin\left\{ t_{1},\ldots,t_{n-1}\right\} .
\]
In particular, for every $\left\{ s_{1},\ldots,s_{n-1}\right\} \neq\left\{ t_{1},\ldots,t_{n-1}\right\} $,
\[
\det\left(\begin{smallmatrix}R_{\widehat{i}}\left(y^{1}\right)\\
Z_{s_{1},\ldots,s_{n-1}}\left(y^{1}\right)
\end{smallmatrix}\right)\equiv0\left(\mbox{mod }2^{2n-2}\right),
\]
since at least two of the rows are equivalent modulo $2^{2n-2}$.
Secondly,
\[
\left(\begin{smallmatrix}R_{i}\left(y^{1}\right)\\
Z_{t_{1},\ldots,t_{n-1}}\left(y^{1}\right)
\end{smallmatrix}\right)\equiv\left(\begin{smallmatrix}R_{\widehat{i}}\left(y^{1}\right)\\
Z_{t_{1},\ldots,t_{n-1}}\left(y^{1}\right)
\end{smallmatrix}\right)\equiv\left(\begin{smallmatrix}3 & 1 & \cdots & 1\\
1 & 3 &  & \vdots\\
\vdots &  & \ddots & 1\\
1 & \cdots & 1 & 3
\end{smallmatrix}\right)\left(\mbox{mod }2^{2n-2}\right).
\]
Since the eigenvalues of the matrix on the right-hand side are $2n+1$
(of multiplicity $1$) and $2$ (of multiplicity $2n-2$), its determinant
equals $\left(2n+1\right)\cdot2^{2n-2}$. In particular,
\[
\det\left(\begin{smallmatrix}R_{\widehat{i}}\left(y^{1}\right)\\
Z_{t_{1},\ldots,t_{n-1}}\left(y^{1}\right)
\end{smallmatrix}\right)\equiv\left(2n+1\right)\cdot2^{2n-2}\left(\mbox{mod }2^{2n-1}\right)\equiv2^{2n-2}\left(\mbox{mod }2^{2n-1}\right).
\]
In formula \ref{eq: A_i(y) is sum of determinants} for $B_{i}\left(y^{1}\right)$,
one summand is equivalent to $2^{2n-2}\left(\mbox{mod }2^{2n-1}\right)$,
and the remaining summands are equivalent to $0\left(\mbox{mod }2^{2n-1}\right)$;
thus, $B_{i}\left(y^{1}\right)\equiv2^{2n-2}\left(\mbox{mod }2^{2n-1}\right)$.

Finally, let $y^{0}\in\mbox{Mat}_{2\l\times\left(2n-1\right)}\left(\ZZ\right)$
be such that $R_{i}\left(y^{0}\right)\equiv R_{\hat{i}}\left(y^{0}\right)\left(\mbox{mod }2^{2n-2}\right)$
for every $i\in\left\{ 1,\ldots,2\l\right\} $. Then for every $i$
and every $t_{1}<\ldots<t_{n-1}\in\left\{ 1,\ldots,\l\right\} \setminus\left\{ i\right\} $,
\[
\det\left(\begin{smallmatrix}R_{\widehat{i}}\left(y^{0}\right)\\
Z_{t_{1},\ldots,t_{n-1}}\left(y^{0}\right)
\end{smallmatrix}\right)\equiv0\left(\mbox{mod }2^{2n-2}\right);
\]
thus, $B_{i}\left(y^{0}\right)\equiv0\left(\mbox{mod }2^{2n-1}\right)$
for every $i\in\left\{ 1,\ldots,2\l\right\} $.
\end{proof}

\begin{prop}
\label{prop: A_i and its twin are maximally related}For every
$i\in\left\{ 1,\ldots,\l\right\} $, $B_{i}\left(x\right)$ and $B_{\hat{i}}\left(x\right)$
are intertwined. \end{prop}

\begin{proof}
Assume the claim holds for $x'\in\mbox{Mat}_{\,2\left(\l-1\right)\times2\left(n-1\right)}$.
Thus, according to Formula \ref{eq: P(x) along last column}, if $x',x''\in\mbox{Mat}_{\,2\left(\l-1\right)\times2\left(n-1\right)}$
differ only by their last column, then $\pitz\left(x'\right)$ and
$\pitz\left(x''\right)$ are intertwined through the set $\left\{ B_{l}\right\} _{1\leq l\leq2\l-2}$.
Since $\pitz\left(M_{i,\hat{i}}^{2n-1}\left(y\right)\right)$ and
$\pitz\left(M_{i,\hat{i}}^{2n-2}\left(y\right)\right)$ differ only
by their last column, they are therefore intertwined through the set
$\left\{ B_{l}\left(M_{i,\hat{i}}^{2n-1,2n-2}\left(y\right)\right)\right\} _{l\in\left\{ 1,\ldots,2\l\right\} \setminus\left\{ i,\hat{i}\right\} }$.
By Formula \ref{eq: A_i(y) expressed through Pitzians}, we conclude
that $B_{i}\left(x\right)$ and $B_{\hat{i}}\left(x\right)$ are intertwined
through the polynomials $\left\{ \pitz\left(M_{i,\hat{i}}^{k}\left(y\right)\right)\right\} _{k=1}^{2n-1}$.
\end{proof}

\begin{proof}[Proof of Theorem \ref{thm: Rectangular matrices variety}]
According to formula \ref{eq: P(x) along last column},
$\inpol\left(x\right)=\pitz\left(x\right)$ is of the form \ref{eq: def of pol Delta},
where by Proposition \ref{prop: A_i and its twin are maximally related}
two of the coefficient-polynomials are intertwined. Lemma \ref{lem: parity conditinos on rectangular coefficients}
asserts that the maximal $\maxexp\in\NN$ such that $2^{\maxexp-1}$
divides every $B_{i}\left(y\right)$ for every odd $y$ is $\maxexp=2n-1$.
Finally, we claim that $\pitz\left(x\right)\equiv0\left(\mbox{mod }2^{2n-1}\right)$
for every odd $x$ in $\mbox{Mat}_{2\l\times2n}\left(\ZZ\right)$.
To this end, we consider the following formula for $\pitz\left(x\right)$
(\cite{Pitzian}):
\[
\pitz\left(x\right)=\sum_{\substack{\left(t_{1},\ldots t_{n}\right)\in\left\{ 1,\ldots,\l\right\} ^{n}\\
t_{1}<\ldots<t_{n}
}
}\det\left(Z_{t_{1},\ldots,t_{n}}\left(x\right)\right);
\]
it asserts that $\pitz\left(x\right)$ is the sum of determinants
of odd $2n\times2n$ matrices, and is therefore divisible by $2^{2n-1}$.
By Theorem \ref{thm: Weak General Thm}, prime matrices are Zariski-dense
in $\vrec_{m}$ if and only if $m\equiv0\left(\mbox{mod }2^{2n-1}\right)$.
\end{proof}

\section{The Permanent Variety\label{sec: non-homogenous varieties}}

Observe that Theorem \ref{thm: General Thm} relies purely on the
combinatorial properties of the defining polynomial $\inpol$ for
the variety $\vgen_{m}$, and in particular does not assume homogeneity
of the variety under a group action. This gives rise to examples of
Zariski-density of prime points in varieties which are not necessarily
homogeneous, such as the hafnian variety mentioned in Remark \ref{rem: hafnian}.
Another such example is the permanent variety.

\begin{defn}
The \emph{permanent} of an $n\times n$ matrix $x=\left(x_{ij}\right)$
is defined as
\[
\mbox{perm}\left(x\right)=\sum_{\sigma\in S_{n}}\prod_{i=1}^{n}x_{i,\sigma\left(i\right)}.
\]

\end{defn}
The permanent of a matrix can be expanded along any row or column;
e.g., an expansion along the $i$-th row is given by
\[
\mbox{perm}\left(x\right)=\sum_{j=1}^{n}x_{i,j}\cdot\mbox{perm}\left(M_{i}^{j}\left(x\right)\right).
\]
The variety of matrices with fixed permanent
\[
\vper_{m}=\left\{ x\in\mbox{Mat}_{n}\left(\F\right)\mid\mbox{perm}\left(x\right)=m\right\}
\]
is not invariant under a group action.

We note that the permanent is to the determinant as the hafnian
is to the Pfaffian: it is obtained form switching all the negative
signs in the determinant polynomial to positive signs. However, while
the congruence condition on the hafnian variety for Zariski-density
of prime points was identical to the one of the Pfaffian, the situation
with the permanent is different from the determinant case.

\def\thethmrpt{\ref{thm: permanent variety}}
\begin{thmrpt}

Let $n\geq3$ and $0\neq m\in\ZZ$. Write $2^{s}-1\leq n<2^{s+1}-1$
for a unique integer $s\geq2$. Then prime matrices are Zariski-dense
in $\vper_{m}$ if and only if
\[
m\equiv\begin{cases}
2^{n-s}\left(\mbox{mod }2^{n-s+1}\right) & \mbox{when }n=2^{s}-1\\
0\left(\mbox{mod }2^{n-s}\right) & \mbox{when }2^{s}-1<n<2^{s+1}-1.
\end{cases}
\]

\end{thmrpt}

The necessity part of Theorem \ref{thm: permanent variety} is slightly
more involved than it was in the previous examples, due to the fact
that the permanent is not invariant under linear actions on the rows
of the matrix. We shall require the following Lemma, whose proof
has been suggested in \cite{MO}.
\begin{lem}
\label{Lemma: perm mod powers of 2}Let $n,s\geq1$ be integers, and
let $x\in\mbox{Mat}_{\, n}\left(\ZZ\right)$ with odd entries.
\begin{enumerate}
\item \label{enu 1: perm mod power of 2}The permanent of $x$ satisfies
\begin{equation}
\perm\left(x\right)\equiv\begin{cases}
2^{n-s}\left(\mbox{mod }2^{n-s+1}\right) & n=2^{s}-1\\
0\left(\mbox{mod }2^{n-s}\right) & 2^{s}\leq n<2^{s+1}-1
\end{cases}.\label{eq: permanent mod powers of 2}
\end{equation}

\item \label{enu 2: both cases occur}Furthermore, when $2^{s}\leq n<2^{s+1}-1$,
$\perm\left(x\right)$ can be congruent to either $0$ or $2^{n-s}$
modulo $2^{n-s+1}$. Both cases occur: there exist odd matrices $x^{0},x^{1}\in\mbox{Mat}_{n}\left(\ZZ\right)$
such that $\perm\left(x^{1}\right)\equiv2^{n-s}\left(\mbox{mod }2^{n-s+1}\right)$
and $\perm\left(x^{0}\right)\equiv0\left(\mbox{mod }2^{n-s+1}\right)$.
\end{enumerate}
\end{lem}
The following fact is instrumental in the proof of Lemma \ref{Lemma: perm mod powers of 2}.
\begin{fact}
\label{lem: 2-power dividing n!}Let $\phi_{2}\left(n\right)$ denote
the highest power of $2$ that divides $n!$, and let $s\geq1$ be
an integer.
\begin{enumerate}
\item If $n=2^{s}-1$ then $\phi_{2}\left(n\right)=n-s$.
\item If $2^{s}\leq n<2^{s+1}-1$, then $\phi_{2}\left(n\right)\geq n-s$.

\end{enumerate}
\end{fact}

This fact is a direct consequence the Legendre
Formula, which states that $\phi_{2}\left(n\right)=n-s_{2}\left(n\right)$,
where $s_{2}\left(n\right)$ is the number of $1$'s in the binary
representation of $n$.

\begin{proof}[Proof of lemma \ref{Lemma: perm mod powers of 2}]
 $ $

\paragraph*{Part \ref{enu 1: perm mod power of 2}}

We prove \ref{eq: permanent mod powers of 2} by induction on $n$.
For $n=1$ we have $n=2^{1}-1$, i.e. $s=1$, and for every odd integer
$x$: $\perm\left(x\right)=x\equiv1\left(\mbox{mod }2\right)$. Let
$J$ denote the $n\times n$ matrix whose all entries are $1$'s.
Since $\perm\left(J\right)=n!$, the claim holds for $J$ according
to Fact \ref{lem: 2-power dividing n!}. Every other odd $n\times n$
matrix is obtained from $J$ by a finite number of steps of the form
``add/subtract $2$ from a given entry of the matrix'', and it is
therefore sufficient to prove that if an odd matrix $x'$ satisfies
\ref{eq: permanent mod powers of 2}, then a matrix $x$ obtained
from $x'$ by adding $\pm2$ to the $\left(i,j\right)$ entry of $x'$,
also satisfies \ref{eq: permanent mod powers of 2}. Recall that $M_{i}^{j}\left(x'\right)$
denotes the matrix obtained from $x'$ by deleting its $i$-th row
and $j$-th column, and observe that:
\begin{equation}
\perm\left(x\right)=\perm\left(x'\right)\pm2\cdot\perm\left(M_{i}^{j}\left(x'\right)\right).\label{eq: perm step}
\end{equation}
By the induction hypothesis, $M_{i}^{j}\left(x'\right)$ satisfies
\ref{eq: permanent mod powers of 2}. We distinguish between three
different cases.
\begin{itemize}
\item If $n=2^{s}-1$, then $n-1=2^{s}-2\in\left[2^{s-1},2^{s}-1\right)$.
Since $x'$ and $M_{i}^{j}\left(x'\right)$ satisfy \ref{eq: permanent mod powers of 2},
we have
\begin{eqnarray*}
\perm\left(x'\right) & \equiv & 2^{n-s}\left(\mbox{mod }2^{n-s+1}\right)\\
\perm\left(M_{i}^{j}\left(x'\right)\right) & \equiv & 0\left(2^{\left(n-1\right)-\left(s-1\right)}\right)\equiv0\left(\mbox{mod }2^{n-s}\right).
\end{eqnarray*}
In particular, by \ref{eq: perm step}:
\[
\perm\left(x\right)\equiv2^{n-s}\left(\mbox{mod }2^{n-s+1}\right)+0\left(\mbox{mod }2^{n-s+1}\right)\equiv2^{n-s}\left(\mbox{mod }2^{n-s+1}\right),
\]
as desired.
\item If $n=2^{s}$, then $n-1=2^{s}-1$. Since $x'$ and $M_{i}^{j}\left(x'\right)$
satisfy \ref{eq: permanent mod powers of 2}, we have
\begin{eqnarray*}
\perm\left(x'\right) & \equiv & 0\left(\mbox{mod }2^{n-s}\right)\\
\perm\left(M_{i}^{j}\left(x'\right)\right) & \equiv & 2^{n-1-s}\left(\mbox{mod }2^{n-s}\right).
\end{eqnarray*}
By \ref{eq: perm step},
\[
\perm\left(x\right)\equiv0\left(\mbox{mod }2^{n-s}\right)+0\left(\mbox{mod }2^{n-s}\right)\equiv0\left(\mbox{mod }2^{n-s}\right),
\]
as desired.
\item If $2^{s}+1\leq n\leq2^{s+1}-2$, then both $n$ and $n-1$ are in
$\left[2^{s},2^{s+1}-2\right]$, and since $x'$ and $M_{i}^{j}\left(x'\right)$
satisfy \ref{eq: permanent mod powers of 2}, we have
\begin{eqnarray*}
\perm\left(x'\right) & \equiv & 0\left(\mbox{mod }2^{n-s}\right)\\
\perm\left(M_{i}^{j}\left(x'\right)\right) & \equiv & 0\left(\mbox{mod }2^{n-s}\right).
\end{eqnarray*}
In particular, by \ref{eq: permanent mod powers of 2}:
\[
\perm\left(x\right)\equiv0\left(\mbox{mod }2^{n-s}\right)+0\left(\mbox{mod }2^{n-s}\right)\equiv0\left(\mbox{mod }2^{n-s}\right),
\]
which concludes the proof of part \ref{enu 1: perm mod power of 2}.
\end{itemize}

\paragraph*{Part \ref{enu 2: both cases occur}.}

Assume first that $n=2^{s}$. Let $y\in\mbox{Mat}_{n-1}\left(\ZZ\right)$
be any odd matrix; by part \ref{enu 1: perm mod power of 2}, $\perm\left(y\right)\equiv2^{n-s-1}\left(\mbox{mod }2^{n-s}\right)$.
Consider the $n\times n$ matrices
\[
x=\left[\begin{array}{c|ccc}
3 & 1 & \cdots & 1\\
\hline 1\\
\vdots &  & y\\
1
\end{array}\right],\, x'=\left[\begin{array}{c|ccc}
1 & 1 & \cdots & 1\\
\hline 1\\
\vdots &  & y\\
1
\end{array}\right].
\]
Both $\perm\left(x\right)$ and $\perm\left(x'\right)$ are congruent
to $0\left(\mbox{mod }2^{n-s}\right)$, by part \ref{enu 1: perm mod power of 2};
we claim that one of them is congruent to $2^{n-s}\left(\mbox{mod }2^{n-s+1}\right)$,
and the other is congruent to $0\left(\mbox{mod }2^{n-s+1}\right)$.
This is due to the fact that
\begin{eqnarray*}
\perm\left(x\right) & = & \perm\left(x'\right)+2\cdot\perm\left(y\right)\\
 & \equiv & \perm\left(x'\right)+2^{n-s}\left(\mbox{mod }2^{n-s+1}\right).
\end{eqnarray*}
This proves the claim of part \ref{enu 2: both cases occur} for $n=2^{s}$,
and we proceed by induction on $n$ in the interval $\left[2^{s},2^{s+1}-1\right]$.
Let $2^{s}+1\leq n<2^{s+1}-1$. By the induction hypothesis, there
exists $\tilde{x}^{1}\in\mbox{Mat}_{n-1}\left(\ZZ\right)$ such that
$\perm\left(\tilde{x}^{1}\right)\equiv2^{n-1-s}\left(2^{n-s}\right)$.
Let
\[
x=\left[\begin{array}{c|ccc}
3 & 1 & \cdots & 1\\
\hline 1\\
\vdots &  & \tilde{x}^{1}\\
1
\end{array}\right],\, x'=\left[\begin{array}{c|ccc}
1 & 1 & \cdots & 1\\
\hline 1\\
\vdots &  & \tilde{x}^{1}\\
1
\end{array}\right].
\]
As before, $\perm\left(x\right)\equiv\perm\left(x'\right)\equiv0\left(\mbox{mod }2^{n-s}\right)$
and
\begin{eqnarray*}
\perm\left(x\right) & = & \perm\left(x'\right)+2\cdot\perm\left(\tilde{x}^{1}\right)\\
 & \equiv & \perm\left(x'\right)+2^{n-s}\left(\mbox{mod }2^{n-s+1}\right);
\end{eqnarray*}
hence one one of the matrices $x$, $x'$ is congruent to $2^{n-s}\left(\mbox{mod }2^{n-s+1}\right)$,
and the other is congruent to $0\left(\mbox{mod }2^{n-s+1}\right)$.
\end{proof}

\begin{proof}[Proof of Theorem \ref{thm: permanent variety}]
For an $n\times n$ matrix of variables $x$, we let $\left(\xi_{1},\ldots,\xi_{n}\right)$
denote the first row of $x$, and let $y$ denote the $\left(n-1\right)\times n$
matrix obtained by deleting the first row of $x$. We write $K_{j}\left(y\right)$
for the permanent of the $\left(n-1\right)\times\left(n-1\right)$
matrix obtained by deleting the $j$-th column of $y$, i.e. $K_{j}\left(y\right)=\mbox{perm}\left(M^{j}\left(y\right)\right)$.
Then, the polynomial $\inpol\left(x\right)=\perm\left(x\right)$ is
of the form \ref{eq: def of pol Delta}
considered in Theorems \ref{thm: General Thm} and \ref{thm: Weak General Thm}:
\[
\perm\left(x\right)=K_{1}\left(y\right)x_{1}+K_{2}\left(y\right)x_{2}+\cdots+K_{n}\left(y\right)x_{n}.
\]
Since the permanent polynomial differs from the determinant only by
the signs, it also has the property that the permanents of two matrices
that differ only by a single row or column are intertwined. Hence
every pair $\left(K_{i}\left(y\right),K_{j}\left(y\right)\right)$
with $i\neq j$ is intertwined.

We are left to verify the parity conditions on $\perm\left(x\right)$
and the coefficients $K_{i}\left(y\right)$. We apply Lemma \ref{Lemma: perm mod powers of 2}
for three different cases.

If $n=2^{s}-1$, then every $M^{j}\left(y\right)$ is an odd square
matrix of order $n-1=2^{s}-2\in\left[2^{s-1},2^{s}-1\right)$ and
in particular
\begin{eqnarray*}
\perm\left(x\right) & \equiv & 2^{n-s}\left(\mbox{mod }2^{n-s+1}\right)\\
K_{j}\left(y\right)=\perm\left(M^{j}\left(y\right)\right) & \equiv & 0\left(2^{\left(n-1\right)-\left(s-1\right)}\right)\equiv0\left(\mbox{mod }2^{n-s}\right).
\end{eqnarray*}
The conditions of Theorem \ref{thm: Weak General Thm} are met  with
$\maxexp=n-s$, $m\equiv2^{n-s}\left(\mbox{mod }2^{n-s+1}\right)$.

If $n=2^{s}$, then every $M^{j}\left(y\right)$ is an odd square
matrix of order $n-1=2^{s}-1$ and in particular
\begin{eqnarray*}
\perm\left(x\right) & \equiv & 0\left(\mbox{mod }2^{n-s}\right)\\
K_{j}\left(y\right)=\perm\left(M^{j}\left(y\right)\right) & \equiv & 2^{n-s-1}\left(2^{n-s}\right).
\end{eqnarray*}
The conditions of Theorem \ref{thm: Weak General Thm} are met  with
$\maxexp=n-s$, $m\equiv0\left(\mbox{mod }2^{n-s}\right)$.

If $2^{s}+1\leq n<2^{s+1}-1$, every $M^{j}\left(y\right)$ is an
odd square matrix of order $n-1\in\left[2^{s},2^{s+1}-2\right)$ and
in particular
\begin{eqnarray*}
\perm\left(x\right) & \equiv & 0\left(\mbox{mod }2^{n-s}\right)\\
K_{j}\left(y\right)=\perm\left(M^{j}\left(y\right)\right) & \equiv & 0\left(2^{n-1-s}\right).
\end{eqnarray*}
The conditions of Theorem \ref{thm: Weak General Thm} are met  with
$\maxexp=n-s$, $m\equiv0\left(\mbox{mod }2^{n-s}\right)$.
\end{proof}

\appendix

\section{Appendix: Proof of Lemma \ref{lem: create 2-coprime}\label{sec: Appendix}}

The goal of this section is to prove Lemma \ref{lem: create 2-coprime}.
We start by recalling Dirichlet's Theorem.
\begin{thm}[Dirichlet theorem on arithmetic progressions]
\label{thm: Dirichlet theorem}Let $\a$ and $\b$ be co-prime integers.
Then there are infinitely many primes in the arithmetic progression
$\left\{ \a+\b\ZZ\right\} $. In other words, there are infinitely
many primes that are congruent to $\alpha$ modulo $\beta$.
\end{thm}
The following is a simple consequence of Dirichlet's Theorem and the
Chinese Remainder Theorem:
\begin{fact}
\label{Dirichlet expansion}Let $\a_{1},\a_{2},\b_{1},\b_{2}\in\mathbb{Z}$
be such that
\[
\gcd\left(\b_{1},\b_{2}\right)=\gcd\left(\a_{1},\b_{1}\right)=\gcd\left(\a_{2},\b_{2}\right)=1.
\]
 Then there are infinitely many primes in the intersection of the
arithmetic progressions $\left\{ \a_{1}+\b_{1}\ZZ\right\} $ and $\left\{ \a_{2}+\b_{2}\ZZ\right\} $. \end{fact}
\begin{proof}
By the Chinese Remainder Theorem there exists $x\in\mathbb{Z}$ in
the intersection of the arithmetic progressions $\left\{ \a_{1}+\b_{1}\ZZ\right\} $
and $\left\{ \a_{2}+\b_{2}\ZZ\right\} $, since $\gcd\left(\b_{1},\b_{2}\right)=1$.
Write
\[
x=\a_{1}+k_{1}\b_{1}=\a_{2}+k_{2}\b_{2}.
\]
If $\ell=\lcm\left(\b_{1},\b_{2}\right)$, then every element in the
arithmetic progression $\left\{ x+\ell\ZZ\right\} $ is contained
in $\left\{ \a_{1}+\b_{1}\ZZ\right\} \cap\left\{ \a_{2}+\b_{2}\ZZ\right\} $.
By Dirichlet theorem, there are infinitely many primes in $\left\{ x+\ell\ZZ\right\} $
if $\gcd\left(x,\ell\right)=1$ --- which is indeed the case, since:
\[
\gcd\left(x,\b_{1}\right)=\gcd\left(\a_{1}+k_{1}\b_{1},\b_{1}\right)=\gcd\left(\a_{1},\b_{1}\right)=1
\]
and
\[
\gcd\left(x,\b_{2}\right)=\gcd\left(\a_{1}+k_{2}\b_{2},\b_{2}\right)=\gcd\left(\a_{2},\b_{2}\right)=1.
\]

\end{proof}

It is well known that the $\gcd$ of a finite set of integers $\a_{1},\ldots,\a_{r}$
can be presented as an integral combination of $\a_{1},\ldots,\a_{r}$;
the content of the following claim is that the integral coefficients
in this combination can be chosen to satisfy some desired congruence.

\begin{claim}
\label{gcd with coprime integers}Let $\a_{1},\ldots,\a_{r}\in\ZZ$
such that $\gcd\left(\a_{1},\ldots,\a_{r}\right)=d$, and let $p\neq2$
a prime that does not divide $d$. Then there exist $t_{1},\ldots,t_{r}\in\ZZ$
such that $p\nmid t_{1},\ldots,p\nmid t_{r}$ and $t_{1}\a_{1}+\ldots+t_{r}\a_{r}=d$.\end{claim}
\begin{proof}
Let $\mu_{1},\ldots,\mu_{r}\in\ZZ$ such that
\[
\mu_{1}\a_{1}+\ldots+\mu_{r}\a_{r}=d.
\]
Since $p\nmid d$, there exists $i\in\left\{ 1,\ldots,r\right\} $
such that $p\nmid\mu_{i}\a_{i}$. Assume $p\nmid\mu_{r}\a_{r}$, namely
$p\nmid\mu_{r}$ and $p\nmid\a_{r}$. If $p\nmid\mu_{i}$ for all
$i=1,\ldots r-1$, we are done. Otherwise, rearrange the indexes such
that $p\mid\mu_{i}$ for $i=1,\ldots k$, for $1\leq k<r$. Consider
the following presentation of $d$ is an integral combination of $\a_{1},\ldots,\a_{r}$:
\[
\left(\mu_{1}+\a_{r}\right)\a_{1}+\ldots+\left(\mu_{k}+\a_{r}\right)\a_{k}+\mu_{k+1}\a_{k+1}+\ldots\mu_{r-1}\a_{r-1}+\left(\mu_{r}-\a_{1}-\ldots-\a_{k}\right)\a_{r}=d.
\]
Note that $p\nmid\mu_{i}+\a_{r}$ for all $i=1,\ldots,k$, since
$p\mid\mu_{i}$ and $p\nmid\a_{r}$. By assumption $p\nmid\mu_{k+1},\ldots,p\nmid\mu_{r-1}$.
If $p\nmid\mu_{r}-\a_{1}-\ldots-\a_{k}$, we are done. Otherwise,
consider the following presentation of $d$ is an integral combination
of $\a_{1},\ldots,\a_{r}$:
\[
\left(\mu_{1}+2\a_{r}\right)\a_{1}+\ldots+\left(\mu_{k}+2\a_{r}\right)\a_{k}+\mu_{k+1}\a_{k+1}+\ldots\mu_{r-1}\a_{r-1}+\left(\mu_{r}-2\a_{1}-\ldots-2\a_{k}\right)\a_{r}=d.
\]
Note that $p\nmid\mu_{i}+2\a_{r}$ for all $i=1,\ldots,k$, since
$p\mid\mu_{i}$ and $p\nmid2\a_{r}$ (as $p\neq2$ and $p\nmid\a_{r}$).
By assumption $p\nmid\mu_{k+1},\ldots,p\nmid\mu_{r-1}$. Finally,
$p\nmid\mu_{r}-2\a_{1}-\ldots-2\a_{k}$; indeed, if $p$ divides both
$\mu_{r}-2\a_{1}-\ldots-2\a_{k}$ and $\mu_{r}-\a_{1}-\ldots-\a_{k}$,
then it must divide $\a_{1}+\ldots+\a_{k}$, and therefore  $p\mid\mu_{r}$,
a contradiction.
\end{proof}

\begin{proof}[Proof of Lemma \ref{lem: create 2-coprime}]
 If $\ga$ is a power of $2$, the claim is trivial; indeed, since
$\gcd\left(q_{i},2^{s_{i}}\right)=1$, the arithmetic progression
$\left\{ q_{i}+2^{s_{i}}\ZZ\right\} $ contains infinitely many primes
(by Dirichlet's Theorem).

Otherwise, let $p$ be an odd prime factor of $\ga$, and $\nu_{p}$
be an integer satisfying
\begin{eqnarray*}
p & \nmid & \nu_{p},\\
p & \nmid & \left(\nu_{p}\cdot\gcd\left(\a_{1},\ldots,\a_{n}\right)+\b\right).
\end{eqnarray*}
Such $\nu_{p}$ exists, because for any integer $\dl_{p}$ for which
\[
\dl_{p}\not\equiv0,\b\left(\mbox{mod }p\right)
\]
we can choose
\[
\nu_{p}\equiv\frac{\dl_{p}-\b}{\gcd\left(\a_{1},\ldots,\a_{n}\right)}\left(\mbox{mod }p\right)
\]
(since $\gcd\left(\a_{1},\ldots,\a_{n}\right)\not\equiv0\left(\mbox{mod }p\right)$),
and then
\[
\nu_{p}\not\equiv0\left(\mbox{mod }p\right)
\]
and
\[
\nu_{p}\cdot\gcd\left(\a_{1},\ldots,\a_{n}\right)+\b\equiv\dl_{p}\not\equiv0\left(\mbox{mod }p\right).
\]
By Claim \ref{gcd with coprime integers}, there exist $n$ integers
$t_{1}^{\left(p\right)},t_{2}^{\left(p\right)},...,t_{n}^{\left(p\right)}$
co-prime to $p$ such that
\[
t_{1}^{\left(p\right)}\a_{1}+t_{2}^{\left(p\right)}\a_{2}+...+t_{n}^{\left(p\right)}\a_{n}=\gcd\left(\a_{1},\ldots,\a_{n}\right).
\]
Then,
\begin{eqnarray*}
\nu_{p}t_{1}^{\left(p\right)}\a_{1}+\nu_{p}t_{2}^{\left(p\right)}\a_{2}+...+\nu_{p}t_{n}^{\left(p\right)}\a_{n}+\b & = & \nu_{p}\cdot\gcd\left(\a_{1},\ldots,\a_{n}\right)+\b\\
 & \not\equiv & 0\left(\mbox{mod }p\right);
\end{eqnarray*}
in particular, if
\[
\left(y_{1},\ldots,y_{n}\right)\equiv\left(\nu_{p}t_{1}^{\left(p\right)},\ldots,\nu_{p}t_{n}^{\left(p\right)}\right)\left(\mbox{mod }p\right),
\]
then
\begin{eqnarray}
 &  & \a_{1}y_{1}+\ldots+\a_{n}y_{n}+\b\nonumber \\
 & \equiv & \a_{1}\cdot\nu_{p}t_{1}^{\left(p\right)}+\a_{2}\cdot\nu_{p}t_{2}^{\left(p\right)}+...+\a_{n}\cdot\nu_{p}t_{n}^{\left(p\right)}+\b\left(\mbox{mod }p\right)\label{eq: p does not divide f(y)}\\
 & \not\equiv & 0\left(\mbox{mod }p\right).\nonumber
\end{eqnarray}
Finally, consider the following set in $\ZZ^{n}$:
\begin{equation}
\left\{ \left(y_{1},\ldots,y_{n}\right)\left|\begin{split} & y_{i}\mbox{ is prime for all \ensuremath{i}}\\
 & y_{i}\equiv q_{i}\left(\mbox{mod }2^{s_{i}}\right)\mbox{ for all \ensuremath{i}}\\
 & \left(y_{1},\ldots,y_{n}\right)\equiv\left(\nu_{p}t_{1}^{\left(p\right)},\ldots,\nu_{p}t_{n}^{\left(p\right)}\right)\left(\mbox{mod }p\right)\mbox{ for all }2\neq p\mid\ga
\end{split}
\right.\right\} .\label{eq: all that modulo}
\end{equation}
Observe that every $y_{i}$ assumes prime values in the following
finite intersection of arithmetic progressions:
\[
\left(\bigcap_{2\neq p\mid\ga}\left\{ \nu_{p}t_{i}^{\left(p\right)}+\mathbb{N}\cdot p\right\} \right)\cap\left\{ q_{i}+\mathbb{N}\cdot2^{s_{i}}\right\} ,
\]
and this intersection contains infinitely many primes, by Fact \ref{Dirichlet expansion};
indeed, $\gcd\left(\nu_{p}t_{i}^{\left(p\right)},p\right)=1$ for
every $p$ and every $i=1,\ldots,n$, $\gcd\left(q_{i},2^{s_{i}}\right)=1$,
 and the elements of $\left\{ \left\{ p:2\neq p\mid\ga\right\} ,\left\{ 2^{s_{i}}\right\} \right\} $
are pairwise co-prime. Hence, every $y_{i}$ assumes infinitely many
values, and the set \ref{eq: all that modulo} is Zariski-dense in
$\F^{n}$.

We conclude the proof by observing that the set \ref{eq: all that modulo}
is contained in the set \ref{eq:create 2-coprime}. First of all,
if $y_{i}\equiv q_{i}\left(\mbox{mod }2^{s_{i}}\right)$ and $q_{i}$
is odd, then in particular $y_{i}$ is odd. By equation \ref{eq: p does not divide f(y)},
every odd prime factor $p$ of $\ga$ does not divide $\linpol$.
\end{proof}

\bibliographystyle{plain}

\end{document}